\documentclass[12pt]{article}
\usepackage{amsmath,amssymb,amsfonts,amsthm,url,xspace,graphicx,psfrag}
\usepackage[pdftitle={Combinatorics of Tripartite Boundary Connections for Trees and Dimers},
            pdfauthor={Richard W. Kenyon and David B. Wilson},
            bookmarks=true,bookmarksopen=true,bookmarksopenlevel=2]{hyperref}
\usepackage[all]{hypcap}

\newcommand{\sref}[1]{\hyperref[#1]{\S~\ref*{#1}}}
\newcommand{\aref}[1]{\hyperref[#1]{Appendix~\ref*{#1}}}
\newcommand{\lref}[1]{\hyperref[#1]{Lemma~\ref*{#1}}}
\newcommand{\tref}[1]{\hyperref[#1]{Theorem~\ref*{#1}}}
\newcommand{\cref}[1]{\hyperref[#1]{Corollary~\ref*{#1}}}
\newcommand{\fref}[1]{\hyperref[#1]{Figure~\ref*{#1}}}
\newcommand{\pref}[1]{\hyperref[#1]{Proposition~\ref*{#1}}}

\def\clap#1{\hbox to 0pt{\hss#1\hss}}
\def\mathllap{\mathpalette\mathllapinternal}

\def\mathllapinternal#1#2{\llap{$\mathsurround=0pt#1{#2}$}}

 \newcommand{\MRhref}[2]{\href{http://www.ams.org/mathscinet-getitem?mr=#1}{MR#2}}

\makeatletter
\def\@strippedMR{}
\def\@scanforMR#1#2#3\endscan{%
  \ifx#1M\ifx#2R\def\@strippedMR{#3}%
  \else\def\@strippedMR{#1#2#3}%
  \fi\fi}
\def\@rst #1 #2other{#1}
\newcommand\MR[1]{\relax\ifhmode\unskip\spacefactor3000 \space\fi
  \@scanforMR#1\endscan
  \MRhref{\expandafter\@rst \@strippedMR other}{\@strippedMR}}
\newcommand\MRs[1]{\relax\ifhmode\unskip\spacefactor3000 \space\fi
  \@scanforMR#1\endscan
  \MRhref{\@strippedMR}{\@strippedMR}}
\makeatother

\newtheorem{theorem}{Theorem}
\numberwithin{theorem}{section}

\newtheorem{lemma}[theorem]{Lemma}

\theoremstyle{definition}
\theoremstyle{definition}
\newcommand{\g}{\mathsf{t}}
\newcommand{\DD}{\mathsf{DD}}
\newcommand{\WB}{\mathsf{WB}}
\newcommand{\BW}{\mathsf{BW}}
\newcommand{\eqz}{\stackrel{\g}{\equiv}}

\newcommand{\R}{\mathbb{R}}

\newcommand{\G}{\mathcal G}

\newcommand{\No}{\mathbf{N}}

\newcommand{\ZDD}{Z^{\DD}}
\newcommand{\ZBW}{Z^{\BW}}
\newcommand{\ZWB}{Z^{\WB}}

\newcommand{\KL}{L}

\newcommand{\eps}{\varepsilon}

\newcommand{\Pfd}{\operatorname{Pfd}}
\newcommand{\Pf}{\operatorname{Pf}}

\newcommand{\old}[1]{}
\renewcommand{\th}{\ensuremath{^{\text{th}}}\xspace}
\newcommand{\st}{\ensuremath{^{\text{st}}}\xspace}

\newcommand{\unc}{\operatorname{uncrossing}}
\newcommand{\tree}{\operatorname{tree}}

\newcommand{\CdV}{MR1652692}
\newcommand{\CdVGV}{MR1371682}
\newcommand{\DGG}{MR1462755}
\newcommand{\CIM}{MR1657214}
\newcommand{\Fomin}{MR1837248}

\newcommand{\Kasteleyn}{MR0253689}

\newcommand{\CS}{MR2097339}

\newcommand{\Pu}{\ddddot\Pr}
\newcommand{\hPr}{\widehat\Pr}

\newcommand{\pu}[1]{\Pu(#1)}

\newcommand{\plt}{\phantom{+}t}
\newcommand{\fm}{\phantom{-}}

\newcommand{\p}{{\mathcal P}^{(\g)}}

\begin{document}
\title{Combinatorics of Tripartite Boundary Connections \\ for Trees and Dimers
}
\author{Richard W. Kenyon\thanks{Brown University, Providence, RI} \and David B. Wilson\thanks{Microsoft Research, Redmond, WA}}
\date{}
\maketitle

\begin{abstract}
  A grove is a spanning forest of a planar graph in which every
  component tree contains at least one of a special subset of vertices
  on the outer face called nodes.  For the natural probability measure
  on groves, we compute various connection probabilities for the nodes
  in a random grove.  In particular, for ``tripartite'' pairings of
  the nodes, the probability can be computed as a Pfaffian in the
  entries of the Dirichlet-to-Neumann matrix (discrete Hilbert
  transform) of the graph.  These formulas generalize the determinant
  formulas given by Curtis, Ingerman, and Morrow, and by Fomin, for
  parallel pairings.  These Pfaffian formulas are used to give exact
  expressions for reconstruction: reconstructing the conductances of a
  planar graph from boundary measurements.  We prove similar theorems
  for the double-dimer model on bipartite planar graphs.
\end{abstract}

\section{Introduction}

In a companion paper \cite{KW:polynomial} we studied two probability
models on finite planar graphs: groves and the double-dimer model.

\subsection{Groves}

Given a finite planar graph and a set of vertices on the outer face,
referred to as nodes, a \textbf{grove} is a spanning forest in which
every component tree contains at least one of the nodes.  A grove
defines a partition of the nodes: two nodes are in the same part if
and only if they are in the same component tree of the grove.
See \fref{grove}.

When the edges of the graph are weighted, one defines a probability
measure on groves, where the probability of a grove is proportional to
the product of its edge weights.  We proved in \cite{KW:polynomial}
that the connection probabilities---the partition of nodes determined
by a random grove---could be computed in terms of certain ``boundary''
measurements.  Explicitly, one can think of the graph as a resistor
network in which the edge weights are conductances.  Suppose the nodes
are numbered in counterclockwise order.  The $L$ matrix, or
Dirichlet-to-Neumann matrix\footnote{Our $L$ matrix is the negative of
  the Dirichlet-to-Neumann matrix of \cite{\CdV}.} (also known as the
response matrix or discrete Hilbert transform), is then the function
$L=(L_{i,j})$ indexed by the nodes, with $L v$ being the vector of net
currents out of the nodes when $v$ is a vector of potentials applied
to the nodes (and no current loss occurs at the internal vertices).
For any partition $\pi$ of the nodes, the probability that a random
grove has partition $\pi$ is $$\Pr(\pi)=\pu{\pi}\Pr(1|2|\cdots|n),$$
where $1|2|\cdots|n$ is the partition which connects no nodes, and
$\pu{\pi}$ is a polynomial in the entries $L_{i,j}$ with integer
coefficients (we think of it as a normalized probability,
$\pu{\pi}=\Pr(\pi)/\Pr(1|2|\cdots|n)$, hence the notation).  In
\cite{KW:polynomial} we showed how the polynomials $\pu{\pi}$ could be
constructed explicitly as integer linear combinations of elementary
polynomials.

\begin{figure}[t!]
\centerline{
\includegraphics{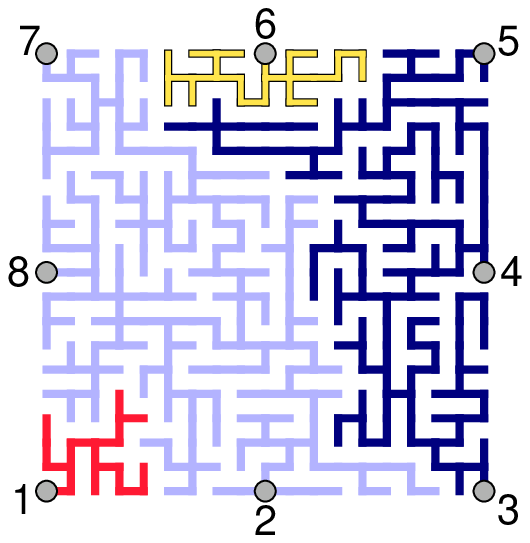}\hfil\includegraphics[viewport=20 -45 11 -29]{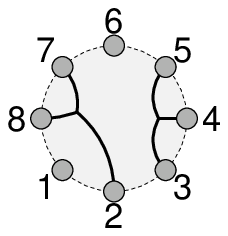}}
\caption{
  A random grove (left) of a rectangular grid with $8$ nodes on the
  outer face.  In this grove there are $4$ trees (each colored
  differently), and the partition of the nodes is $\{\{1\}, \{2,7,8\},
  \{3,4,5\}, \{6\}\}$, which we write as $1|278|345|6$, and illustrate
  schematically as shown on the right.
\label{grove}}
\end{figure}

For certain partitions~$\pi$, however, there is a simpler formula for
$\pu{\pi}$: for example, Curtis, Ingerman, and Morrow \cite{\CIM}, and
Fomin \cite{\Fomin}, showed that for certain partitions~$\pi$,
$\pu{\pi}$ is a determinant of a submatrix of $L$.  We generalize
these results in several ways.

Firstly, we give an interpretation (\sref{minors}) of every minor
of~$L$ in terms of grove probabilities.  This is analogous to the
all-minors matrix-tree theorem \cite{chaiken}
\cite[pg.~313 Ex.~4.12--4.16, pg.~295]{chen}, except that the
matrix entries are entries of the response matrix rather than edge
weights, so in fact the all-minors matrix-tree theorem is a special
case.

Secondly, we consider the case of \textbf{tripartite}
partitions~$\pi$ (see \fref{rgb-tripartite}),
 showing that the grove probabilities $\pu{\pi}$ can be written
as the Pfaffian of an antisymmetric matrix derived from the $L$
matrix.  One motivation for studying tripartite partitions is the work
of Carroll and Speyer \cite{MR2097339} and Petersen and Speyer
\cite{MR2144860} on so-called Carroll-Speyer groves (\fref{csgrove})
which arose in their study of the cube recurrence.  Our tripartite
groves directly generalize theirs.  See \sref{CSgroves}.

\begin{figure}[b!]
\includegraphics{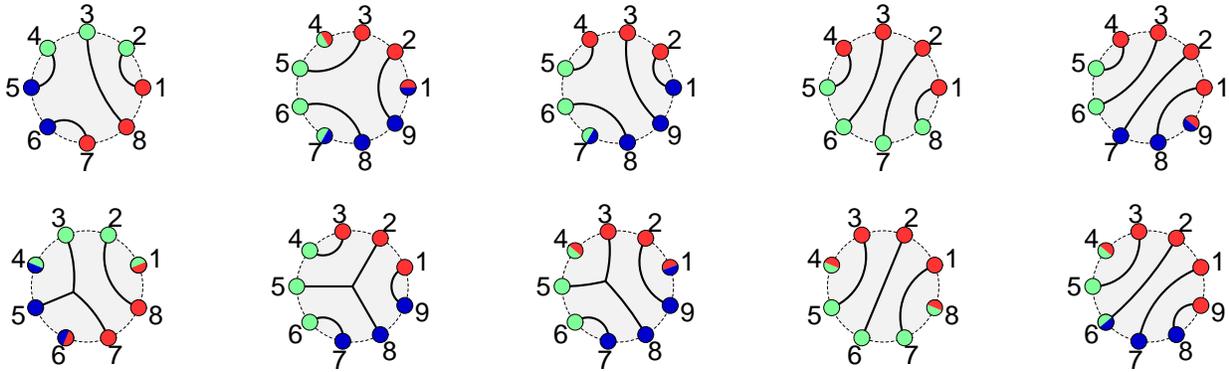}
\caption{
  Illustration of tripartite partitions.  The two partitions in each
  column are duals of one another.  The nodes come in three colors,
  red, green, and blue, which are arranged contiguously on the outer
  face; a node may be split between two colors if it occurs at the
  transition between these colors.  Assuming the number of nodes of
  each color (where split nodes count as half) satisfies the triangle
  inequality, there is a unique noncrossing partition with a minimal
  number of parts in which no part contains nodes of the same color.
  This partition is called the tripartite partition, and is
  essentially a pairing, except that there may be singleton nodes
  (where the colors transition), and there may be a (unique) part of
  size three.  If there is a part of size three, we call the partition
  a \textbf{tripod}.  If one of the color classes is empty (or the triangle
  inequality is tight), then the partition is the ``parallel crossing''
  studied in \cite{\CIM} and \cite{\Fomin}.
\label{rgb-tripartite}}
\end{figure}

A third motivation is the conductance reconstruction problem.  Under
what circumstances does the response matrix ($L$ matrix), which is a
function of boundary measurements, determine the conductances on the
underlying graph?  This question was studied in \cite{\CIM, \CdV,
  \CdVGV}.  Necessary and sufficient conditions are given in
\cite{\CdVGV} for two planar graphs on $n$ nodes to have the same
response matrix.  In \cite{\CdV} it was shown which matrices arise as
response matrices of planar graphs.  Given a response matrix~$L$
satisfying the necessary conditions, in \sref{reconstruction} we use
the tripartite grove probabilities to give explicit formulas for the
conductances on a standard graph whose response matrix is $L$.  This
question was first solved in \cite{\CIM}, who gave an algorithm for
recursively computing the conductances, and was studied further in
\cite{card-muranaka,russell}.  In contrast, our formulas are
explicit.

\subsection{Double-dimer model}

A number of these results extend to another probability model, the
double-dimer model on bipartite planar graphs, also discussed in
\cite{KW:polynomial}.

Let $\G$ be a finite bipartite graph\footnote{Bipartite means that the
  vertices can be colored black and white such that adjacent vertices
  have different colors.} embedded in the plane with a set~$N$ of $2n$
distinguished vertices (referred to as nodes) which are on the outer
face of $\G$ and numbered in counterclockwise order.  One can consider
a multiset (a subset with multiplicities) of the edges of $\G$ with the
property that each vertex except the nodes is the endpoint of exactly
two edges, and the nodes are endpoints of exactly one edge in the
multiset.  In other words, it is a subgraph of degree~$2$ at the
internal vertices, degree~$1$ at the nodes, except for possibly having
some doubled edges.  Such a configuration is called a double-dimer
configuration; it will connect the nodes in pairs.

If edges of $\G$ are weighted with positive real weights, one defines a
probability measure in which the probability of a configuration is a
constant times the product of weights of its edges (and doubled edges
are counted twice), times $2^\ell$ where $\ell$ is the number of loops
(doubled edges do not count as loops).

We proved in \cite{KW:polynomial} that the connection
probabilities---the matching of nodes determined by a random
configuration---could be computed in terms of certain boundary
measurements.

Let $\ZDD(\G,\No)$ be the weighted sum of all double-dimer
configurations.  Let $\G^\BW$ be the subgraph of $\G$ formed by
deleting the nodes except the ones that are black and odd or white and
even, and let $\G^\BW_{i,j}$ be defined as $\G^\BW$ was, but with
nodes $i$ and $j$ included in $\G^\BW_{i,j}$ if and only if they were
not included in $\G^\BW$.

A dimer cover, or perfect matching, of a graph is a set of edges whose
endpoints cover the vertices exactly once.  When the graph has weighted
edges, the weight of a dimer configuration is the product of its edge
weights.  Let $\ZBW$ and $\ZBW_{i,j}$ be the weighted sum of dimer
configurations of $\G^\BW$ an $\G^\BW_{i,j}$, respectively, and define
$\ZWB$ and $\ZWB_{i,j}$ similarly but with the roles of black and
white reversed.  Each of these quantities can be computed via
determinants, see \cite{\Kasteleyn}.

One can easily show that $\ZDD = \ZBW \ZWB$; this is essentially
equivalent to Ciucu's graph factorization theorem \cite{ciucu}.
(The two dimer configurations in \fref{ddimer-dimer} are on the
graphs $\G^\BW$ and $\G^\WB$.)
The variables that play the role of $L_{i,j}$ in groves are defined by
$$X_{i,j} = \ZBW_{i,j} / \ZBW.$$
\begin{figure}[htbp]
\includegraphics[height=0.31\textwidth]{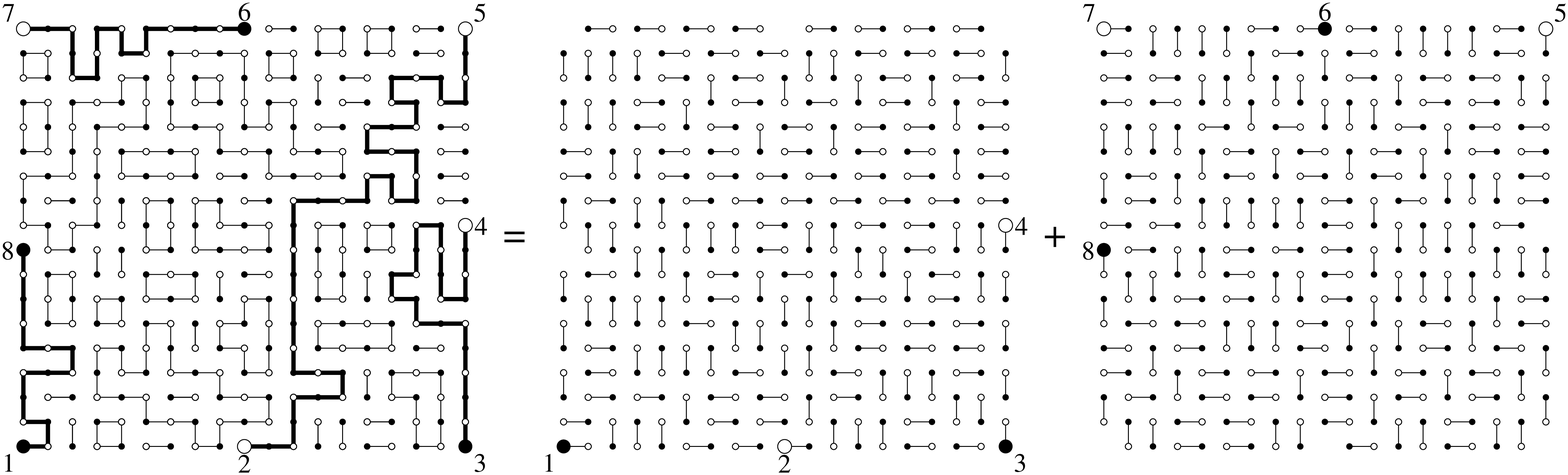}
\caption{A double-dimer configuration is a union of two dimer configurations.
\label{ddimer-dimer}}
\end{figure}

We showed in \cite{KW:polynomial} that for each matching~$\pi$, the
normalized probability $\hPr(\pi)=\Pr(\pi)\ZWB/\ZBW$ that a random
double-dimer configuration connects nodes in the matching~$\pi$, is an
integer polynomial in the quantities $X_{i,j}$.

In the present paper, we show in \tref{dd-tripartite} that
when~$\pi$ is a tripartite pairing, that is, the nodes are divided
into three consecutive intervals around the boundary and no node is
paired with a node in the same interval, $\hPr(\pi)$ is a determinant
of a matrix whose entries are the $X_{i,j}$'s or $0$.

\subsection{Conductance reconstruction}

Recall that an electrical transformation of a resistor network is a
local rearrangement of the type shown in \fref{electricalmoves}.
These transformations do not change the response matrix of the graph.
\cite{\CdVGV} showed that a planar graph with $n$ nodes can be reduced,
using electrical transformations, to a standard graph~$\Sigma_n$
(shown in \fref{stdgraphs} for $n$ up to $6$), or a minor of one of
these graphs (obtained from $\Sigma_n$ by deletion/contraction of
edges).

\begin{figure}[b!]
\psfrag{=}[bc][Bc][1][0]{$\Leftrightarrow$}
\psfrag{a}[bc][Bc][1][0]{$a$}
\psfrag{b}[bc][Bc][1][0]{$b$}
\psfrag{c}[bc][Bc][1][0]{$c$}
\psfrag{a+b}[bc][Bc][1][0]{$a+b$}
\psfrag{ab/(a+b)}[bc][Bc][1][0]{\small $ab/(a+b)$}
\psfrag{ab/(a+b+c)}[bc][Bc][1][0]{$\frac{ab}{a+b+c}$}
\psfrag{ac/(a+b+c)}[bc][Bc][1][0]{$\frac{ac}{a+b+c}$}
\psfrag{bc/(a+b+c)}[tc][Bc][1][0]{$\frac{bc}{a+b+c}$}
\includegraphics{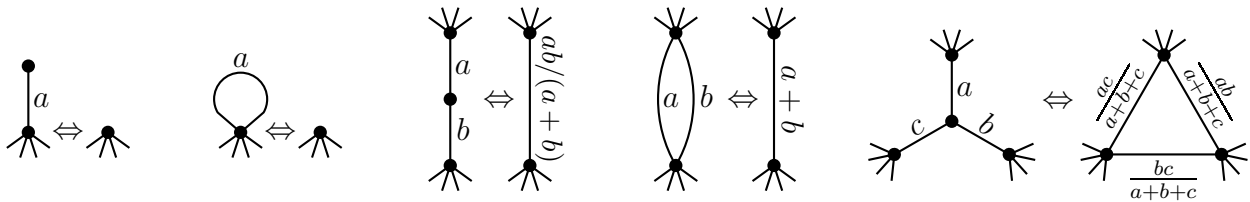}
\caption{Local graph transformations that preserve the
  electrical response matrix of the graph; the edge weights are the
  conductances.  These transformations also preserve the connection
  probabilities of random groves, though some of the transformations
  scale the weighted sum of groves.  Any planar graph with $n$ nodes
  can be transformed to a minor of the ``standard graph''~$\Sigma_n$
  (\fref{stdgraphs}) via these transformations \cite{\CdVGV}.
\label{electricalmoves}}
\end{figure}

\begin{figure}[htbp]
\psfrag{G1}[tc][tc][1][0]{$\Sigma_1$}
\psfrag{G2}[tc][tc][1][0]{$\Sigma_2$}
\psfrag{G3}[tc][tc][1][0]{$\Sigma_3$}
\psfrag{G4}[tc][tc][1][0]{$\Sigma_4$}
\psfrag{G5}[tc][tc][1][0]{$\Sigma_5$}
\psfrag{G6}[tc][tc][1][0]{$\Sigma_6$}
\centerline{\includegraphics{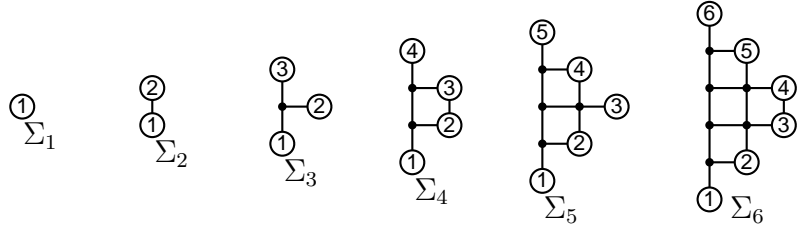}}
\caption{Standard graphs $\Sigma_n$ with $n$ nodes for $n$ up to $6$.
\label{standardgraphs}\label{stdgraphs}}
\end{figure}

In particular the response matrix of any planar graph on $n$ nodes is
the same as that for a minor of the
standard graph~$\Sigma_n$ (with certain
conductances).  \cite{\CdV} computed which matrices occur as response
matrices of a planar graph.  \cite{\CIM} showed how to reconstruct
recursively the edge conductances of $\Sigma_n$ from the response
matrix, and the reconstruction problem was also studied in
\cite{card-muranaka} and \cite{russell}.  Here we give an explicit
formula for the conductances as ratios of Pfaffians of matrices
derived from the $L$ matrix and its inverse.  These Pfaffians are
irreducible polynomials in the matrix entries (\tref{irreducible}),
so this is in some sense the minimal expression for
the conductances in terms of the $L_{i,j}$.

\section{Background}

Here we collect the relevant facts from \cite{KW:polynomial}.
\enlargethispage{12pt}

\subsection{Partitions}

We assume that the nodes are labeled $1$ through $n$ counterclockwise
around the boundary of the graph~$\G$.  We denote a partition of the
nodes by the sequences of connected nodes, for example $1|234$ denotes
the partition consisting of the parts $\{1\}$ and $\{2,3,4\}$, i.e.,
where nodes $2$, $3$, and $4$ are connected to each other but not to
node $1$.  A partition is crossing if it contains four items $a<b<c<d$
such that $a$ and $c$ are in the same part, $b$ and $d$ are in the same
part, and these two parts are different.  A partition is planar if and only if
it is non-crossing, that is,
it can be represented by arranging the items in order on a
circle, and placing a disjoint collection of connected sets in the
disk such that items are in the same part of the partition when they
are in the same connected set.    For example $13|24$ is the only non-planar
partition on $4$ nodes.

\subsection{Bilinear form and projection matrix}

Let $W_n$ be the vector space consisting of formal linear combinations
of partitions of $\{1,2,\dots,n\}$.  Let $U_n\subset W_n$ be the subspace
consisting of formal linear combinations of planar partitions.

On $W_n$ we define a bilinear form: if $\tau$ and $\sigma$ are
partitions, $\langle \tau,\sigma\rangle_\g$ takes value $1$ or $0$ and
is equal to~$1$ if and only if the following two conditions are
satisfied:
\begin{enumerate}
\item The number of parts of $\tau$ and $\sigma$ add up to $n+1$.
\item The transitive closure of the
relation on the nodes defined by the union of $\tau$ and $\sigma$ has a single
equivalence class.
\end{enumerate}
For example $\langle 123|4,24|1|3\rangle_\g=1$ but $\langle
12|34,12|3|4\rangle_\g=0.$ (We write the subscript $\g$ to distinguish
this form from ones that arise in the double-dimer model in
\sref{ddsection}.)

This form, restricted to the subspace $U_n$, is essentially the
``meander matrix'', see \cite{KW:polynomial,\DGG}, and has non-zero
determinant.  Hence the bilinear form is non-degenerate on $U_n$.  We
showed in \cite{KW:polynomial}, Proposition 2.6, that $W_n$ is the
direct sum of $U_n$ and a subspace~$K_n$ on which
$\langle,\!\rangle_\g$ is identically zero.  In other words, the rank
of $\langle,\!\rangle_\g $ is the $n$\th Catalan number~$C_n$, which
is the dimension of $U_n$.  Projection to $U_n$ along the kernel~$K_n$
associates to each partition~$\tau$ a linear combination of planar
partitions.  The matrix of this projection is called~$\p$.  It has
integer entries \cite{KW:polynomial}.  Observe that $\p$ preserves the
number of parts of a partition: each non-planar partition with $k$
parts projects to a linear combination of planar partitions with $k$
parts (this follows from condition 1 above).

\subsection{Equivalences}

The rows of the projection matrix~$\p$ determine the crossing probabilities,
see \tref{KW1mainthm} below.  In this section we give tools for
computing columns of~$\p$.

We say two elements of $W_n$ are equivalent ($\eqz$) if their
difference is in $K_n$, that is, their inner product with any
partition is equal.  We have, for example,
\begin{lemma} \label{cross}
  $1|234 + 2|134 + 3|124 + 4|123 \eqz 12|34 + 13|24 + 14|23$
\end{lemma}
which is another way of saying that
$$\p(13|24)=1|234 + 2|134 + 3|124 + 4|123-12|34-14|23.$$

This lemma, together with the following two equivalences, will allow
us to write any partition as an equivalent sum of planar partitions.
That is, it allows us to compute the columns of $\p$.

\begin{lemma} \label{newpart}
  Suppose $n\geq2$, $\tau$ is a partition of $1,\dots,n-1$, and $\tau
  \eqz \sum_\sigma \alpha_\sigma \sigma$.  Then
  $$\tau|n \eqz \sum_\sigma \alpha_\sigma \sigma|n.$$
\end{lemma}

If $\tau$ is a partition of $1,\dots, n-1$, we can insert $n$ into
the part containing item $j$ to get a partition of $1,\dots,n$.

\begin{lemma} \label{split}
  Suppose $n\geq2$, $\tau$ is a partition of $1,\dots,n-1$, $1\leq j<n$,
  and $\tau \eqz \sum_\sigma \alpha_\sigma \sigma$.  Then
 $$\text{\rm [$\tau$ with $n$
  inserted into $j$'s part]} \eqz \sum_\sigma \alpha_\sigma \text{\rm [$\sigma$ with
    $n$ inserted into $j$'s part]}.$$
\end{lemma}

One more lemma is quite useful for computations.

\begin{lemma}[\relax{\cite[Lemma~4.1]{KW:polynomial}}]
\label{drop-singletons}
  If a planar partition $\sigma$ contains only singleton and doubleton parts,
  and $\sigma'$ is the partition obtained from $\sigma$ by deleting all
  the singleton parts, then the rows of the matrices $\p$ for $\sigma$ and
  $\sigma'$ are equal, in the sense that they have the same non-zero entries
  (when the columns are matched accordingly by deleting
  the corresponding singletons).
\end{lemma}

The above lemmas can be used to recursively rewrite a
non-planar partition $\tau$ as an equivalent linear combination
of planar partitions. As a simple example, to reduce $13|245$,
start with the equation from \lref{cross} and, using
\lref{split}, adjoin a $5$
to every part containing $4$, yielding
$$13|245 \equiv 1|2345+2|1345+3|1245+45|123-12|345-145|23.
$$

\subsection{Connection probabilities}

For a partition $\tau$ on $1,\dots,n$ we define
\begin{equation}\label{Ltau}
\KL_\tau = \sum_F\prod_{\text{$\{i,j\} \in F$}} L_{i,j},
\end{equation}
where the sum is over those spanning forests $F$ of the complete graph on
$n$ vertices $1,\dots,n$ for which trees of $F$ span the parts of $\tau$.

This definition makes sense whether or not the partition $\tau$ is planar.
For example, $\KL_{1|234} = L_{2,3} L_{3,4} + L_{2,3}L_{2,4} + L_{2,4} L_{3,4}$
and $\KL_{13|24} = L_{1,3} L_{2,4}$.

Recall that $\pu{\sigma}=\Pr(\sigma)/\Pr(\unc).$

\begin{theorem}[Theorem 1.2 of \cite{KW:polynomial}]\label{KW1mainthm}
 $$\pu{\sigma} = \sum_\tau \p_{\sigma,\tau} \KL_\tau.$$
 \end{theorem}

\section{Tripartite pairing partitions}
\label{dualtripodL}

Recall that a tripartite partition is defined by three circularly
contiguous sets of nodes $R$, $G$, and $B$, which represent the red
nodes, green nodes, and blue nodes (a node may be split between two
color classes), and the number of nodes of the different colors
satisfy the triangle inequality.  In this section we deal with
tripartite partitions in which all the parts are either doubletons or
singletons.  (We deal with tripod partitions in the next section.)  By
\lref{drop-singletons} above, in fact additional singleton nodes could
be inserted into the partition at arbitrary locations, and the
$L$-polynomial for the partition would remain unchanged.  Thus we lose
no generality in assuming that there are no singleton parts in the
partition, so that it is a \textbf{tripartite pairing} partition.  This
assumption is equivalent to assuming that each node has only one color.

\begin{theorem}\label{tripartite}
  Let $\sigma$ be the tripartite pairing partition defined by circularly
  contiguous sets of nodes $R$, $G$, and $B$, where $|R|$, $|G|$, and $|B|$
  satisfy the triangle inequality.  Then
  $$\pu{\sigma} = \Pf \begin{bmatrix}
     0 & \fm L_{R,G} &\fm L_{R,B}\\
   -L_{G,R} &  0 &\fm L_{G,B}\\
   -L_{B,R} &-L_{B,G} &  0
\end{bmatrix}.$$
\end{theorem}
Here $L_{R,G}$ is the submatrix of $L$ whose columns are the red nodes
and rows are the green nodes.  Similarly for $L_{R,B}$ and $L_{G,B}$.
Also recall that the Pfaffian $\Pf M$ of an antisymmetric $2n\times 2n$
matrix~$M$
is a square root of the determinant of $M$, and is a polynomial in the
matrix entries:
\begin{equation} \label{Pf}
 \Pf M = \sum_{\substack{\text{permutations $\pi$}\\ \pi_1<\pi_2,\dots,\pi_{2n-1}<\pi_{2n} \\ \pi_1<\pi_3<\cdots<\pi_{2n-1}}} (-1)^\pi M_{\pi_1,\pi_2} M_{\pi_3,\pi_4} \cdots M_{\pi_{2n-1},\pi_{2n}}
= \pm \sqrt{\det M},
\end{equation}
where the sum can be interpreted as a sum over pairings of $\{1,\ldots,2n\}$,
since any of the $2^n n!$ permutations associated with a pairing
$\{\{\pi_1,\pi_2\},\ldots,\{\pi_{2n-1},\pi_{2n}\}\}$ would give the
same summand.

In \aref{dualtripodR} there is a corresponding formula for
tripartite pairings in terms of the matrix~$R$ of pairwise resistances
between the nodes.

Observe that we may renumber the nodes while preserving their cyclic
order, and the above Pfaffian remains unchanged: if we move the last
row and column to the front, the sign of the Pfaffian changes, and
then if we negate the (new) first row and column so that the entries
above the diagonal are non-negative, the Pfaffian changes sign again.

As an illustration of the theorem, we have
\begin{align}
\raisebox{-24pt}{\includegraphics{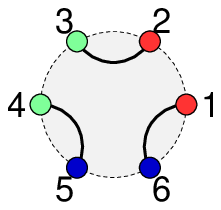}} \ \ \
\pu{16|23|45}
=&\label{pfaffexample}
\Pf\left[\begin{matrix}
 0         &  0         &\fm L_{1,3} &\fm L_{1,4} &\fm L_{1,5} &\fm L_{1,6} \\
 0         &  0         &\fm L_{2,3} &\fm L_{2,4} &\fm L_{2,5} &\fm L_{2,6} \\
-L_{1,3} & -L_{2,3} &  0         &  0         &\fm L_{3,5} &\fm L_{3,6} \\
-L_{1,4} & -L_{2,4} &  0         &  0         &\fm L_{4,5} &\fm L_{4,6} \\
-L_{1,5} & -L_{2,5} & -L_{3,5} & -L_{4,5} &  0         &  0         \\
-L_{1,6} & -L_{2,6} & -L_{3,6} & -L_{4,6} &  0         &  0
\end{matrix}\right] \\
 =& L_{1,3}L_{2,5}L_{4,6} - L_{1,3}L_{2,6}L_{4,5}
 - L_{1,4}L_{2,5}L_{3,6} + L_{1,4}L_{2,6}L_{3,5} \notag \\
 &- L_{1,5}L_{2,3}L_{4,6} + L_{1,5}L_{2,4}L_{3,6}
 + L_{1,6}L_{2,3}L_{4,5} - L_{1,6}L_{2,4}L_{3,5}. \notag
\end{align}

Note that when one of the colors (say blue) is absent, the Pfaffian
becomes a determinant (in which the order of the green vertices is
reversed).  This bipartite determinant special case was proved by
Curtis, Ingerman, and Morrow \cite[Lemma~4.1]{\CIM} and Fomin
\cite[Eqn.~4.4]{\Fomin}.  See \sref{minors} for a (different)
generalization of this determinant special case.

\begin{proof}[Proof of \tref{tripartite}]
  From \tref{KW1mainthm}
  we are interested in computing the non-planar partitions
 $\tau$  (columns of $\p$) for which $\p_{\sigma,\tau}\neq 0$.

  When we project $\tau$, if $\tau$ has singleton parts, its image must consist
  of planar partitions having those same singleton parts, by the lemmas
  above: all the transformations preserve the singleton parts.
  Since $\sigma$ consists of only doubleton parts, because of the
 on the number of parts, $\p_{\sigma,\tau}$ is non-zero only
  when $\tau$ contains only doubleton
  parts.  Thus in \lref{cross} we may use the abbreviated transformation rule
\begin{equation} \label{pairing-rule}
   13|24 \to -14|23 -12|34.
\end{equation}
  Notice that if we take any crossing pair of indices, and apply this rule to it,
  each of the two resulting partitions has fewer crossing pairs than
  the original partition, so repeated application of this rule is
  sufficient to express $\tau$ as a linear combination of planar
  partitions.

  If a non-planar partition~$\tau$ contains a monochromatic part,
  and we apply Rule~\eqref{pairing-rule} to
  it, then because the colors are contiguous, three of the above
  vertices are of the same color, so both of the resulting partitions
  contain a monochromatic part.  When doing the transformations, once
  there is a monochromatic doubleton, there always will be one, and
  since $\sigma$ contains no such monochromatic doubletons, we may
  restrict attention to columns~$\tau$ with no monochromatic
  doubletons.

  When applying Rule~\eqref{pairing-rule} since there are only three
  colors, some color must appear twice.  In one of the resulting
  partitions there must be a monochromatic doubleton, and we may
  disregard this partition since it will contribute $0$.  This allows
  us to further abbreviate the uncrossing transformation rule:
    $$\operatorname{red}_1 x | \operatorname{red}_2 y \to - \operatorname{red}_1 y | \operatorname{red}_2 x,$$
  and similarly for green and blue.  Thus for any partition $\tau$
  with only doubleton parts, none of which are monochromatic, we have
  $\p_{\sigma,\tau}=\pm1$, and otherwise $\p_{\sigma,\tau}=0$.

  If we consider the Pfaffian of the matrix
  $$\begin{bmatrix}
     0 & \fm L_{R,G} & \fm L_{R,B}\\
   -L_{G,R} &  0 & \fm L_{G,B}\\
   -L_{B,R} &-L_{B,G} &  0
  \end{bmatrix},$$
  each monomial corresponds to a monomial in the $L$-polynomial of $\sigma$, up
  to a possible sign change that may depend on the term.

  Suppose that the nodes are numbered from $1$ to $2n$ starting with the
  red ones, continuing with the green ones, and finishing with the
  blue ones.  Let us draw the linear chord diagram corresponding to
  $\sigma$.  Pick any chord, and move one of its endpoints to be
  adjacent to its partner, while maintaining their relative order.
  Because the chord diagram is non-crossing, when doing the move, an
  integer number of chords are traversed, so an even number of
  transpositions are performed.  We can continue this process until
  the items in each part of the partition are adjacent and in sorted order, and the
  resulting permutation will have even sign.  Thus in the above
  Pfaffian, the term corresponding to $\sigma$ has positive sign, i.e., the
  same sign as the
  $\sigma$ monomial in $\sigma$'s $L$-polynomial.

  Next we consider other pairings $\tau$, and show by induction on the
  number of transpositions required to transform $\tau$ into $\sigma$,
  that the sign of the $\tau$ monomial in $\sigma$'s $L$-polynomial
  equals the sign of the $\tau$ monomial in the Pfaffian.  Suppose
  that we do a swap on $\tau$ to obtain a pairing $\tau'$ closer to
  $\sigma$.  In $\sigma$'s $L$ polynomial, $\tau$ and $\tau'$ have
  opposite sign.  Next we compare their signs in the Pfaffian.  In the
  parts in which the swap was performed, there is at least one
  duplicated color (possibly two duplicated colors).  If we implement
  the swap by transposing the items of the same color, then the items
  in each part remain in sorted order, and the sign of the permutation
  has changed, so $\tau$ and $\tau'$ have opposite signs in the
  Pfaffian.

  Thus $\sigma$'s $L$-polynomial is the Pfaffian of the above matrix.
\end{proof}

\section{Tripod partitions}
\label{tripodL}

In this section we show how to compute $\pu{\sigma}$ for tripod
partitions~$\sigma$, i.e., tripartite partitions~$\sigma$ in which one
of the parts has size three.  The three lower-left panels of
\fref{rgb-tripartite} and the left panels of \fref{tripod} and
\fref{csgrove} show some examples.

\subsection{Via dual graph and inverse response matrix} \label{viadual}

For every tripod partition~$\sigma$, the dual partition~$\sigma^*$ is
also tripartite, and contains no part of size three.  As a
consequence, we can compute the probability $\pu{\sigma}$ when
$\sigma$ is a tripod in terms of a Pfaffian in the entries of the
response matrix~$L^*$ of the dual graph~$\G^*$:
$$ \pu{\sigma} = \frac{\Pr(\text{$\sigma$ in $\G$})}{\Pr(\text{$1|2|\cdots|n$ in $G$})} = \frac{\Pr(\text{$\sigma^*$ in $\G^*$})}{\Pr(\text{$1|2|\cdots|n$ in $\G^*$})} \frac{\Pr(\text{$12\cdots n$ in $\G$})}{\Pr(\text{$1|2|\cdots|n$ in $\G$})}.$$
The last ratio in the right is known to be an $(n-1)\times(n-1)$ minor
of $L$ (see e.g., \sref{minors}); it remains to express the
matrix~$L^*$ in terms of~$L$.

Let $i'$ be the node of the dual graph which is located between the
nodes $i$ and $i+1$ of $\G$.

\begin{lemma}
  The entries of $L^*$ are related to the entries of $L$ as follows:
  $$L^*_{i',j'} = (\delta_i-\delta_{i+1})L^{-1}(\delta_j-\delta_{j+1}).$$
\end{lemma}

Here even though $L$ is not invertible, the vector
$\delta_j-\delta_{j+1}$ is in the image of $L$ and $\delta_i-\delta_{i+1}$
is perpendicular to the kernel of $L$, so the above expression is well defined.

\begin{proof}
From \cite[Proposition~2.9]{KW:polynomial}, we have
$$L_{i',j'}^*=\frac12(R_{i,j}+R_{i+1,j+1}-R_{i,j+1}-R_{i+1,j}),$$
where $R_{i,j}$ is the resistance between nodes $i$ and $j$.
From \cite[Proposition~A.2]{KW:polynomial},
$$R_{i,j}=-(\delta_i-\delta_j)^TL^{-1}(\delta_i-\delta_j).$$
The result follows.
\end{proof}

\subsection{Via Pfaffianiod}

In \sref{viadual} we saw how to compute $\pu{\sigma}$ for tripartite
partitions~$\sigma$.  It is clear that the formula given there is
a rational function of the $L_{i,j}$'s, but from \tref{KW1mainthm},
we know that it is in fact a polynomial in the $L_{i,j}$'s.
Here we give the polynomial.

We saw in \sref{dualtripodL} that the Pfaffian was relevant to
tripartite pairing partitions, and that this was in part because the
Pfaffian is expressible as a sum over pairings.  For tripod partitions
(without singleton parts), the relevant matrix operator resembles a
Pfaffian, except that it is expressible as a sum over near-pairings,
where one of the parts has size~$3$, and the remaining parts have
size~$2$.  We call this operator the \textbf{Pfaffianoid}, and
abbreviate it~$\Pfd$.  Analogous to \eqref{Pf}, the
Pfaffianoid of an antisymmetric $(2n+1)\times(2n+1)$ matrix~$M$ is
defined by
\begin{equation} \label{Pfd}
 \Pfd M = \sum_{\substack{\text{permutations $\pi$}\\ \pi_1<\pi_2,\dots,\pi_{2n-3}<\pi_{2n-2} \\ \pi_{2n-1}<\pi_{2n+1} \\ \pi_1<\pi_3<\cdots<\pi_{2n-3}}} (-1)^\pi M_{\pi_1,\pi_2} M_{\pi_3,\pi_4} \cdots M_{\pi_{2n-3},\pi_{2n-2}} \times M_{\pi_{2n-1},\pi_{2n}} M_{\pi_{2n},\pi_{2n+1}},
\end{equation}
where the sum can (almost) be interpreted as a sum over
near-pairings (one tripleton and rest doubletons) of $\{1,\ldots,2n+1\}$,
since for any permutation associated with the near-pairing
$\{\{\pi_1,\pi_2\},\ldots,\{\pi_{2n-3},\pi_{2n-2}\},\{\pi_{2n-1},\pi_{2n},\pi_{2n+1}\}\}$,
the summand only depends on the order of the items in the tripleton part.

The sum-over-pairings formula for the Pfaffian is fine as a
definition, but there are more computationally efficient ways (such as
Gaussian elimination) to compute the Pfaffian.  Likewise, there are
more efficient ways to compute the Pfaffianoid than the above
sum-over-near-pairings formula.  For example, we can write
\begin{equation}
  \Pfd M = \sum_{1\leq a<b<c\leq 2n+1} (-1)^{a+b+c} (M_{a,b} M_{b,c} + M_{b,c} M_{a,c} + M_{a,c} M_{a,b}) \Pf [M\smallsetminus\{a,b,c\}],
\end{equation}
where $M\smallsetminus\{a,b,c\}$ denotes the matrix~$M$ with rows and
columns $a$, $b$, and $c$ deleted.  It is also possible to represent
the Pfaffianoid as a double-sum of Pfaffians.

The tripod probabilities can written as a Pfaffianoid in
the $L_{i,j}$'s as follows:
\begin{theorem}\label{thm:tripod}
  Let $\sigma$ be the tripod partition without singletons defined by circularly
  contiguous sets of nodes $R$, $G$, and $B$, where $|R|$, $|G|$, and $|B|$
  satisfy the triangle inequality.  Then
  $$\pu{\sigma} = 
(-1)^\text{\rm sum of items in $\sigma$'s tripleton part}
\Pfd \begin{bmatrix}
     0 & \fm L_{R,G} &\fm L_{R,B}\\
   -L_{G,R} &  0 &\fm L_{G,B}\\
   -L_{B,R} &-L_{B,G} &  0
\end{bmatrix}.$$
\end{theorem}

The proof of \tref{thm:tripod} is similar in nature to the proof of
\tref{tripartite}, but there are more cases, so we give the proof
in \aref{sec:pfd}.

Unlike the situation for tripartite partitions, here we cannot appeal
to \lref{drop-singletons} to eliminate singleton parts from a tripod
partition, since \lref{drop-singletons} does not apply when there is a
part with more than two nodes.  However, any nodes in singleton parts
of the partition can be split into two monochromatic nodes of
different color, one of which is a leaf.  The response matrix of the
enlarged graph is essentially the same as the response matrix of the
original graph, with some extra rows and columns for the leaves which
are mostly zeroes.  \tref{thm:tripod} may then be applied to this
enlarged graph to compute $\pu{\sigma}$ for the original graph.

\section{Irreducibility}

\begin{theorem} \label{irreducible}
  For any non-crossing partition $\sigma$, $\pu{\sigma}$ is an
  irreducible polynomial in the $L_{i,j}$'s.
\end{theorem}

By looking at the dual graph, it is a straightforward consequence
of \tref{irreducible} that $\Pr(\sigma)/\Pr(12\cdots n)$ is an
irreducible polynomial on the pairwise resistances.  In contrast,
for the double-dimer model, the polynomials $\hPr(\sigma)$ sometimes
factor (the first, second, and fourth examples in~\sref{ddsection} factor).

\begin{proof}[Proof of \tref{irreducible}]
Suppose that $\pu{\sigma}$ factors into $\pu{\sigma} = P_1 P_2$ where $P_1$
and $P_2$ are polynomials in the $L_{i,j}$'s.  Because $\pu{\sigma} =
\sum_\tau \p_{\sigma,\tau} \KL_\tau$ and each $\KL_\tau$ is
multilinear in the $L_{i,j}$'s, we see that no variable $L_{i,j}$
occurs in both polynomials $P_1$ and $P_2$.

Suppose that for distinct vertices $i,j,k$, the variables $L_{i,j}$
and $L_{i,k}$ both occur in $\pu{\sigma}$, but occur in different
factors, say $L_{i,j}$ occurs in $P_1$ while $L_{i,k}$ occurs in
$P_2$.  Then the product $\pu{\sigma}$ contains monomials divisible by
$L_{i,j} L_{i,k}$.  If we consider one such monomial, then the
connected components (with edges given by the indices of the variables
of the monomial) define a partition $\tau$ for which
$\p_{\sigma,\tau}\neq 0$ and for which $\tau$ contains a part
containing at least three distinct items $i$, $j$, and $k$.
Then $L_\tau$ contains $L_{j,k}$, so
$L_{j,k}$ also occurs in one of $P_1$ or $P_2$, say (w.l.o.g.) that it occurs in
$P_1$.  Because $L_\tau$ contains monomials divisible by
$L_{i,j}L_{j,k}$, so does $\pu{\sigma}$, and hence $P_1$ must contain
monomials divisible by $L_{i,j}L_{j,k}$.  But then $P_1 P_2$ would
contain monomials divisible by $L_{i,j} L_{i,k} L_{j,k}$, but
$\pu{\sigma}$ contains no such monomials, a contradiction, so in fact
$L_{i,j}$ and $L_{i,k}$ must occur in the same factor of $\pu{\sigma}$.

If we consider the graph which has an edge $\{i,j\}$ for each variable
$L_{i,j}$ of $\pu{\sigma}$, we aim to show that the graph is connected
except possibly for isolated vertices; it will then follow that
$\pu{\sigma}$ is irreducible.

We say that two parts $Q_1$ and $Q_2$ of a non-crossing partition $\sigma$
are mergeable if the partition $\sigma\setminus\{Q_1,Q_2\}\cup\{Q_1\cup Q_2\}$
is non-crossing. It suffices, to complete the proof, to show
that if $Q_1$ and $Q_2$ are mergeable parts of $\sigma$,
then $\pu{\sigma}$ contains $L_{a,c}$
for some $a\in Q_1$ and $c\in Q_2$.

Suppose $Q_1$ and $Q_2$ are mergeable parts of $\sigma$ that both have
at least two items.  When the items are listed in cyclic order, say
that $a$ is the last item of $Q_1$ before $Q_2$, $b$ is the first item
of $Q_2$ after $Q_1$, $c$ is the last item of $Q_2$ before $Q_1$, and
$d$ is the first item of $Q_1$ after $Q_2$.  Let $\tau$ be the
partition formed from $\sigma$ by swapping $c$ and $d$.  Let
$\sigma^*=\sigma\setminus\{Q_1,Q_2\}$, and let $A=Q_1\setminus\{d\}$
and $B=Q_2\setminus\{c\}$.  Then $\sigma = \sigma^* \cup
\{A\cup\{d\},B\cup\{c\}\}$ and $\tau = \sigma^* \cup \{A\cup\{c\},
B\cup\{d\}\}$.  Then
\begin{multline*}
\tau \to -\sigma
 -(\sigma^* \cup\{A\cup B,\{c,d\}\}) \\
 +(\sigma^* \cup\{A\cup B\cup\{d\},\{c\}\})
 +(\sigma^* \cup\{A\cup B\cup\{c\},\{d\}\}) \\
 +(\sigma^* \cup\{A, B\cup\{c,d\}\})
 +(\sigma^* \cup\{A\cup\{c,d\}, B\}).
\end{multline*}
Each of the partitions on the right-hand side is non-crossing, so
$\p_{\sigma,\tau}=-1$, so in particular $L_{a,c}$ occurs in
$\pu{\sigma}$.

Now suppose that $\sigma$ contains a singleton part $\{a\}$ and
another part $Q_2$ containing at least three items $b$, $c$, $d$,
where $b$, $c$, and $d$ are the first, second, and last items of the
part $Q_2$ as viewed from item $a$.  Let
$\sigma^*=\sigma\setminus\{\{a\},Q_2\}$ and $C=Q_2\setminus\{b,d\}$.
Let $\tau$ be the partition $$\tau = \sigma^* \cup
\{\{a\}\cup C,\{b,d\}\}.$$ Now
\begin{multline*}
\tau \to \sigma
 +(\sigma^* \cup \{\{b\},\{a,d\}\cup C\})
 +(\sigma^* \cup \{\{d\},\{a,b\}\cup C\})
 +(\sigma^* \cup \{C,\{a,b,d\}\})\\
 -(\sigma^* \cup \{\{b\}\cup C,\{a,d\}\})
 -(\sigma^* \cup \{\{a,b\},\{d\}\cup C\})
.
\end{multline*}
The second, third, fourth, fifth, and sixth terms on the RHS
contribute nothing to $\p_{\sigma,\tau}$ because their restrictions to
the intervals $[b,b]$, $[d,d]$, $(b,d)$, $[b,d)$, and $(b,d]$
respectively are planar and do not agree with $\sigma$.  Thus
$\p_{\sigma,\tau}=1$, and hence $L_{a,c}$ occurs in $\pu{\sigma}$.

Finally, if $\sigma$ contains singleton parts but no parts with at
least three items, then $\pu{\sigma}$ is formally identical to the
polynomial $\pu{\sigma^*}$ where $\sigma^*$ is the partition with the
singleton parts removed from $\sigma$, and we have already shown above
that the polynomial $\pu{\sigma^*}$ is irreducible.
\end{proof}

\section{Tripartite pairings in the double-dimer model}
\label{ddsection}

In this section we prove a determinant formula for the tripartite
pairing in the double-dimer model.
\begin{theorem} \label{dd-tripartite}
Suppose that the nodes are contiguously colored red, green, and
blue (a color may occur zero times),
and that $\sigma$ is the (unique) planar pairing in which like colors are
not paired together.  Let $\sigma_i$ denote the item that $\sigma$ pairs
with item~$i$.  We have
  $$\hPr(\sigma) = \det[1_{\text{\rm $i$, $j$ colored differently}} X_{i,j}]^{i=1,3,\dots,2n-1}_{j=\sigma_1,\sigma_3,\dots,\sigma_{2n-1}}.$$
\end{theorem}
For example,
\begin{align*}
\raisebox{-24pt}{\includegraphics{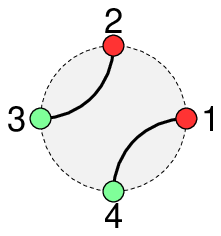}}\ \ \ \ \ 
 \hPr(\,{}^1_4\!\mid\!{}^3_2\,) =&
\begin{vmatrix}
  X_{1,4} & 0 \\
  0 & X_{3,2}
\end{vmatrix}\\
\intertext{(this first example formula is essentially Theorems~2.1 and~2.3 of \cite{kuo-condense}, see also \cite{kuo-pfaff} for a generalization different from the one considered here)
}
\raisebox{-24pt}{\includegraphics{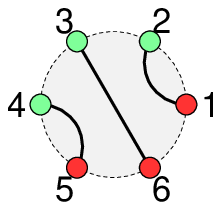}}\ \ \ \ \ 
 \hPr(\,{}^1_2\!\mid\!{}^3_6\!\mid\!{}^5_4\,)=&
\begin{vmatrix}
 X_{1,2} & 0       & X_{1,4} \\
 0       & X_{3,6} & 0       \\
 X_{5,2} & 0       & X_{5,4} \\
\end{vmatrix}
\end{align*}
\begin{align*}
\raisebox{-24pt}{\includegraphics{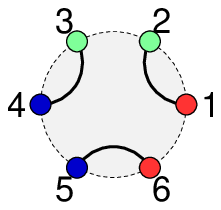}}\ \ \ \ \ 
 \hPr(\,{}^1_2\!\mid\!{}^3_4\!\mid\!{}^5_6)=& \begin{vmatrix}
 X_{1,2} & X_{1,4} & 0       \\
 0       & X_{3,4} & X_{3,6} \\
 X_{5,2} & 0       & X_{5,6}
\end{vmatrix}\\
\raisebox{-24pt}{\includegraphics{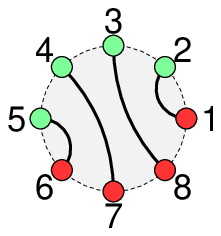}}\ \ \ \ \ 
\hPr(\,{}^1_2\!\mid\!{}^3_8\!\mid\!{}^5_6\!\mid\!{}^7_4\,)
=& \begin{vmatrix}
X_{1,2} & 0       & 0       & X_{1,4} \\
0       & X_{3,8} & X_{3,6} & 0       \\
0       & X_{5,8} & X_{5,6} & 0       \\
X_{7,2} & 0       & 0       & X_{7,4} \\
 \end{vmatrix}\\
\raisebox{-24pt}{\includegraphics{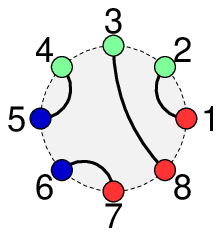}}\ \ \ \ \ 
\hPr(\,{}^1_2\!\mid\!{}^3_8\!\mid\!{}^5_4\!\mid\!{}^7_6\,)
=&\begin{vmatrix}
X_{1,2} & 0       & X_{1,4} & X_{1,6} \\
0       & X_{3,8} & 0       & X_{3,6} \\
X_{5,2} & X_{5,8} & X_{5,4} & 0       \\
X_{7,2} & 0       & X_{7,4} & X_{7,6} \\
\end{vmatrix}
\end{align*}

\begin{proof}
Recall our theorem from \cite{KW:polynomial}, which shows how to compute
the ``$X$'' polynomials for the double-dimer model in terms of the
``$L$'' polynomials for groves:

\begin{theorem}[Kenyon-Wilson '06]
  If a partition $\sigma$ contains only doubleton parts, then if we make
  the following substitutions to the grove partition polynomial $\pu{\sigma}$:
$$
 L_{i,j} \to \begin{cases} 0 &\text{if $i$ and $j$ have the same parity} \\
     (-1)^{(|i-j|-1)/2} X_{i,j} &\text{otherwise}\end{cases}
$$
then the result is the double-dimer pairing polynomial $\hPr(\sigma)$,
when we interpret $\sigma$ as a pairing.
\end{theorem}

Thus our Pfaffian formula for tripartite groves in terms of the
$L_{i,j}$'s immediately gives a Pfaffian formula for the double-dimer
model.  For the double-dimer tripartite formula there are
node parities as well as colors (recall that the graph is bipartite).
Rather than take a Pfaffian of the full
matrix, we can take the determinant of the odd/even submatrix,
whose rows are indexed by red-even, green-even, and blue-even
vertices, and whose columns are indexed by red-odd, green-odd, and
blue-odd vertices.
For example, when computing the probability
$$\hPr(\,{}^1_8\!\mid\!{}^3_4\!\mid\!{}^5_6\!\mid\!{}^7_2\,),$$
nodes $1$, $2$, and $3$ are red, $4$ and $5$ are green, and $6$,
$7$, and $8$ are blue; the odd nodes are black, and the even ones
are white.  The $L$-polynomial is
$$
\Pf
\begin{bmatrix}
0 &       0 &       0 & \fm L_{1,4} & \fm L_{1,5} & \fm L_{1,6} & \fm L_{1,7} & \fm L_{1,8}\\
0 &       0 &       0 & \fm L_{2,4} & \fm L_{2,5} & \fm L_{2,6} & \fm L_{2,7} & \fm L_{2,8}\\
0 &       0 &       0 & \fm L_{3,4} & \fm L_{3,5} & \fm L_{3,6} & \fm L_{3,7} & \fm L_{3,8}\\
-L_{4,1} &-L_{4,2} &-L_{4,3} &       0 &       0 & \fm L_{4,6} & \fm L_{4,7} & \fm L_{4,8}\\
-L_{5,1} &-L_{5,2} &-L_{5,3} &       0 &       0 & \fm L_{5,6} & \fm L_{5,7} & \fm L_{5,8}\\
-L_{6,1} &-L_{6,2} &-L_{6,3} &-L_{6,4} &-L_{6,5} &       0 &       0 &       0\\
-L_{7,1} &-L_{7,2} &-L_{7,3} &-L_{7,4} &-L_{7,5} &       0 &       0 &       0\\
-L_{8,1} &-L_{8,2} &-L_{8,3} &-L_{8,4} &-L_{8,5} &       0 &       0 &       0\\
\end{bmatrix}
$$
Next we do the substitution $L_{i,j}\to0$ when $i+j$ is even, and
reorder the rows and columns so that the odd nodes are listed first.
Each time we swap a pair of rows and do the same swap on the
corresponding pair of columns, the sign of the Pfaffian changes by
$-1$.  Since there are $2n$ nodes the number of swaps is $n(n-1)/2$.
If $n$ is congruent to $0$ or $1\bmod 4$ the sign does not change, and
otherwise it does change.  After the rows and columns have been sorted
by their parity, the matrix has the form
$$
\begin{bmatrix} 0 & \pm L_{O,E} \\ \mp L_{E,O} & 0 \end{bmatrix},
$$
where $O$ represents the odd nodes and $E$ the even nodes, and
where the individual signs are $+$ if the odd node has smaller index
than the even node, and $-$ otherwise.  The Pfaffian of this matrix is
just the determinant of the upper-right submatrix, times the sign of
the permutation $1,n+1,2,n+2,\dots,n,2n$, which is $(-1)^{n(n-1)/2}$.
This sign cancels the above $(-1)^{n(n-1)/2}$ sign.  In this example we get
$$
\det
\begin{bmatrix}
       0 & L_{1,4} & L_{1,6} & L_{1,8}\\
       0 & L_{3,4} & L_{3,6} & L_{3,8}\\
-L_{5,2} &       0 & L_{5,6} & L_{5,8}\\
-L_{7,2} &-L_{7,4} &       0 &       0\\
\end{bmatrix}.
$$
Next we do the $L_{i,j}\to (-1)^{(|i-j|-1)/2}X_{i,j}$ substitution.
The $i,j$ entry of this matrix is $(-1)^{i>j}L_{i,j}$.  Each time that
$i$ or $j$ are incremented or decremented by $2$, the
$(-1)^{(|i-j|-1)/2}$ sign will flip, unless the $(-1)^{i>j}$ sign also
flips.  After the substitution, the signs of the $X_{i,j}$ are
staggered in a checkerboard pattern.  If we then multiply every other
row by $-1$ and every other column by $-1$, the determinant is
unchanged and all the signs are $+$.  In the example we get
$$
\det
\begin{bmatrix}
       0 &-X_{1,4} & \fm X_{1,6} &-X_{1,8}\\
       0 & \fm X_{3,4} &-X_{3,6} &\fm X_{3,8}\\
\fm X_{5,2} &       0 & \fm X_{5,6} &-X_{5,8}\\
-X_{7,2} &\fm X_{7,4} &       0 &       0\\
\end{bmatrix}
=\det
\begin{bmatrix}
       0 & X_{1,4} & X_{1,6} & X_{1,8}\\
       0 & X_{3,4} & X_{3,6} & X_{3,8}\\
 X_{5,2} &       0 & X_{5,6} & X_{5,8}\\
 X_{7,2} & X_{7,4} &       0 &       0\\
\end{bmatrix}.
$$
There is then a global sign of $(-1)^\sigma$ where the sign of the
pairing~$\sigma$ is the sign of sign of the permutation of the even
elements when the parts are arranged in increasing order of their odd
parts.  In our example, the sign of
$\,{}^1_8\!\mid\!{}^3_4\!\mid\!{}^5_6\!\mid\!{}^7_2\,$ is the sign of
$8462$, which is $-1$.  This global sign may be canceled by reordering
the columns in this order, i.e., so that the pairing $\sigma$ can be
read in the indices along the diagonal of the matrix, which for our
example is
\begin{align*}
&
\det
\begin{bmatrix}
 X_{1,8} & X_{1,4} & X_{1,6} &       0 \\
 X_{3,8} & X_{3,4} & X_{3,6} &       0 \\
 X_{5,8} &       0 & X_{5,6} & X_{5,2} \\
       0 & X_{7,4} &       0 & X_{7,2} \\
\end{bmatrix}.
\end{align*}
\end{proof}

\section{Reconstruction on the ``standard graphs'' \texorpdfstring{$\Sigma_n$}{}}\label{reconstruction}

Given a planar resistor network, can we determine (or ``reconstruct'')
the conductances on the edges from boundary measurements, that is,
from the entries in the $L$ matrix?

While reconstruction is not possible in general, each planar graph is
equivalent, through a sequence of electrical transformations, to a
graph on which generically the conductances can be reconstructed.  Let
$\Sigma_n$ denote the standard graph on $n$ nodes, illustrated in
\fref{stdgraphs} for $n$ up to $6$.  Every circular planar graph
with $n$ nodes is electrically equivalent to a minor of a standard
graph $\Sigma_n$.

Here we will use the Pfaffian formulas to give explicit formulas for
reconstruction on standard graphs.  For minors of standard graphs, the
conductances can be computed by taking limits of the formulas for
standard graphs.

Curtis, Ingerman and Morrow \cite{\CIM} gave a recursive construction
to compute conductances for standard graphs from the $L$-matrix.  Card
and Muranaka \cite{card-muranaka} give another way.
Russell \cite{russell} shows how to recover the conductances, and
shows that they are rational functions of $L$-matrix entries.  However
the solution is sometimes given parametrically, as a solution to
polynomial constraints, even when graph is recoverable.

For a vertex $v\in\Sigma_n$ we define $\pi_v$ to be the tripod partition of
the nodes indicated in \fref{tripod}, with a single triple
connection joining the nodes $v_{\to}$ horizontally across from $v$ and
the two nodes $v_\uparrow,v_\downarrow$ vertically located from $v$ (in the same
column as $v$), and the remaining nodes joined in nested pairs between
$v_\to$ and $v_\uparrow$, $v_\uparrow$ and $v_\downarrow$, and $v_\downarrow$ and $v_\to$ (with up to two
singletons if $v_\to,v_\uparrow$ and/or $v_\to,v_\downarrow$ have an odd number of nodes
between them).

Similarly, for a bounded face $f$ of $\Sigma_n$ define $\pi_f$ to be the
tripartite partition of the nodes indicated in
\fref{dualtripod}.  It has three nested sequences of pairwise
connections (with two of the nested sequences going to the NE and SE, possibly
terminating in singletons).  We think of the unbounded face as
containing many ``external faces,'' each consisting of a unit square
which is adjacent to an edge of $\Sigma_n$.  For each of these external
faces, we define $\pi_f$ in the same manner as for internal faces.
For the external faces $f$ on the left of $\Sigma_n$, the ``left-going'' nested
sequence of $\pi_f$ is empty.  For the other external faces $f$, the
partition $\pi_f$ is $(1,n|2,n-1|\cdots)$, independent of~$f$.
\begin{figure}[htbp]
\psfrag{n1}[bl][Bl][1][0]{$v_\to$}
\psfrag{n2}[bc][Bc][1][0]{$v_\uparrow$}
\psfrag{n3}[tc][tc][1][0]{$v_\downarrow$}
\psfrag{f}[bc][Bc][1][0]{$f$}
\psfrag{o}[bc][Bc][1][0]{$v$}
\centerline{\includegraphics[height=0.38\textwidth]{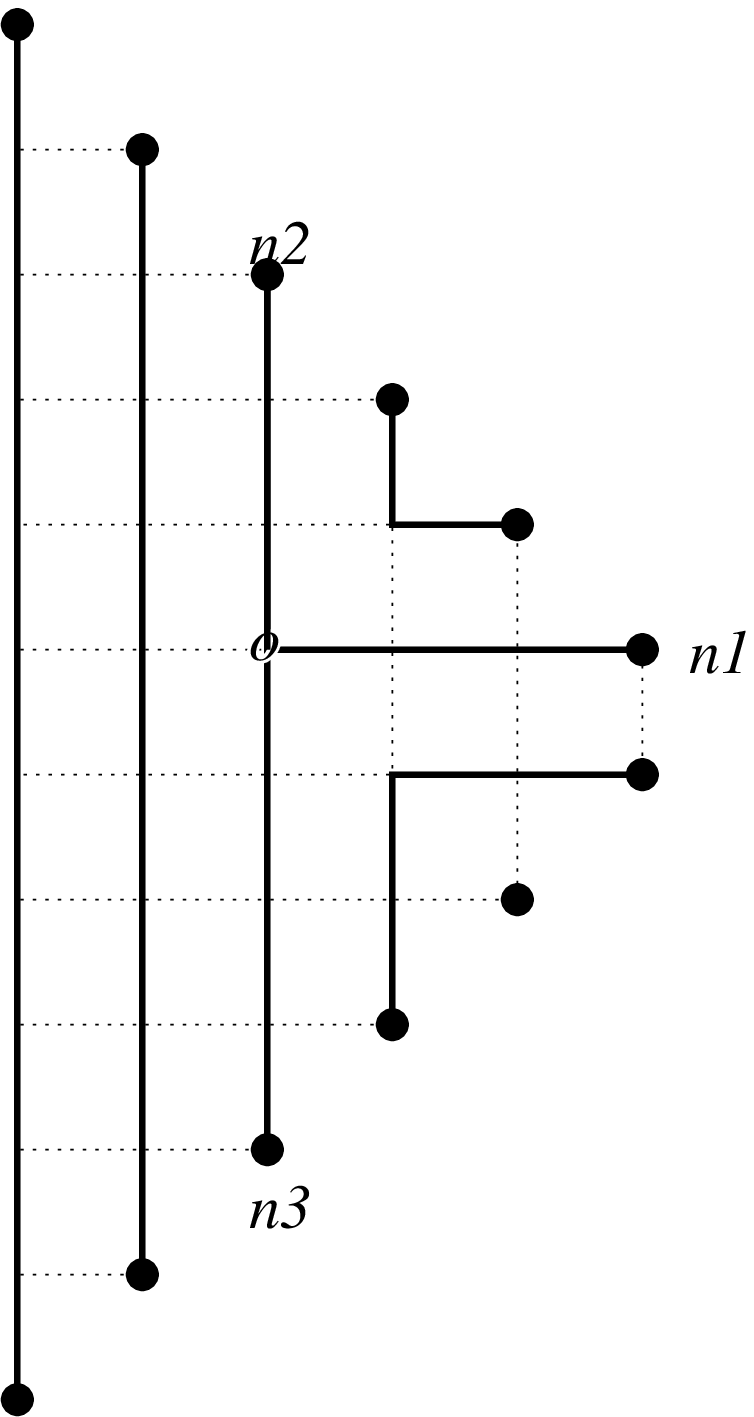}
\hfil
\includegraphics[height=0.38\textwidth]{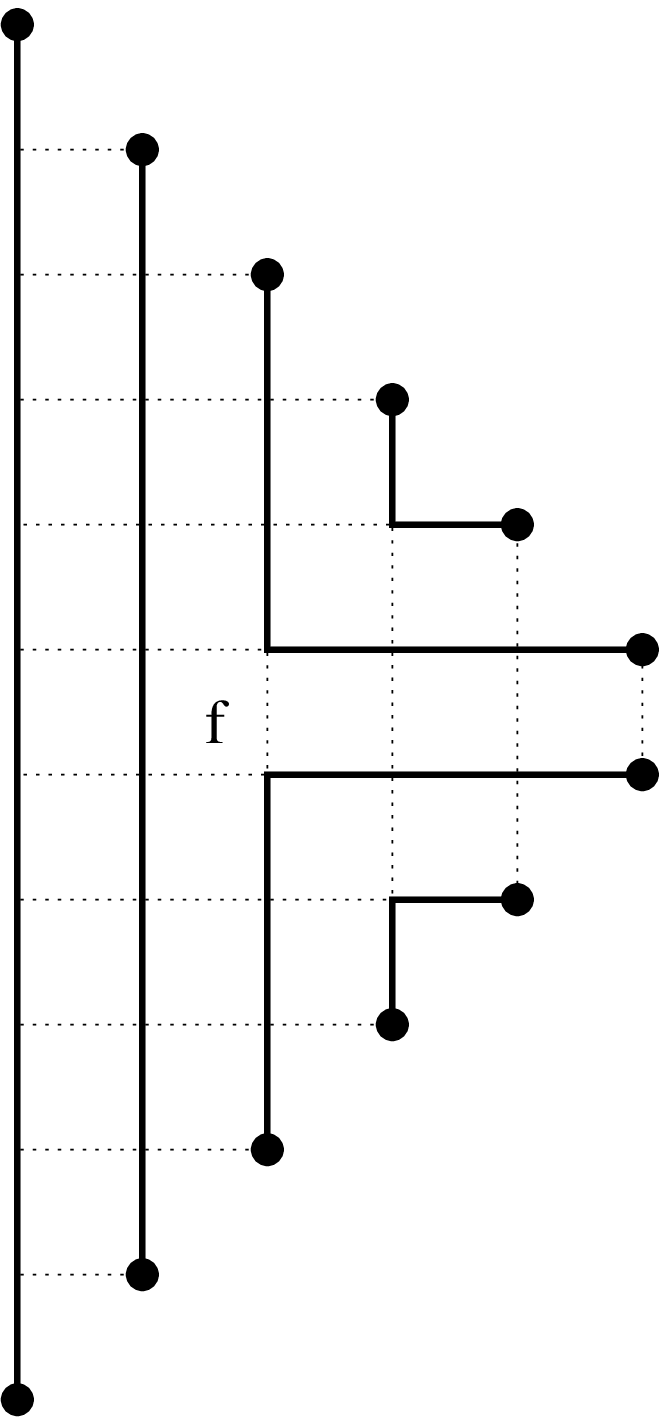}}
\caption{Tripod partition (left) and tripartite partition (right) on the standard graph~$\Sigma_n$.
  \label{tripod}  \label{dualtripod}}
\end{figure}

Observe that for the standard graphs $\Sigma_n$, there is only one
grove of type~$\pi_v$ or of type~$\pi_f$.  Let $a_e$ denote the
conductance of edge~$e$ in $\Sigma_n$.  Each $Z_{\pi_v}$ and
$Z_{\pi_f}$ is a monomial in these conductances $a_e$.  To simplify
notation we write $Z_v=Z_{\pi_v}$ and $Z_f=Z_{\pi_f}.$

Each conductance $a_e$ can be written in terms of the $Z_v$ and $Z_f$:

\begin{lemma}\label{afromL}
  For an edge $e$ of the standard graph~$\Sigma_n$, let $v_1$ and
  $v_2$ be the endpoints of $e$, and let $f_1$ and $f_2$ be the faces
  bounded by $e$.  We have
  $$a_e = \frac{Z_{v_1}Z_{v_2}}{Z_{f_1}Z_{f_2}}.$$
\end{lemma}

\begin{proof} A straightforward inspection of the various cases.
\end{proof}

Combining this lemma with the results of Sections~\ref{tripodL}
and~\ref{dualtripodL}, we can write each edge conductance $a_e$ as a
rational function in the $L_{i,j}$'s.  Since the $Z_v$'s and $Z_f$'s
are irreducible by \tref{irreducible}, this formula is the simplest
rational expression for the $a_e$'s in terms of the $L_{i,j}$'s.

\section{Minors of the response matrix}\label{minors}

We have the following interpretation of the minors of $L$.

\begin{theorem}\label{Lminorthm}
  For general graphs (not necessarily planar), suppose that $A$, $B$,
  $C$, and $D$ are pairwise disjoint sets of nodes such that $|A|=|B|$
  and $A\cup B\cup C\cup D$ is the set of all nodes.  Then the
  determinant of $L_{(A\cup C),(B\cup C)}$ is given by
$$
\det [L_{i,j}]^{i=a_1,\dots,a_{|A|},c_1,\dots,c_{|C|}}_{j=b_1,\dots,b_{|B|},c_1,\dots,c_{|C|}}= \frac{\sum_\pi (-1)^{\pi} \Pr[a_1,b_{\pi(1)}|\cdots|a_{|A|},b_{\pi(|A|)}|d_1|\cdots|d_{|D|}]}{\Pr[\unc]}
$$
where the nodes of $C$ may appear in any of the above parts.
\end{theorem}

In \aref{dualtripodR}, equation~\eqref{Rform},
there is a corresponding formula in terms of
the pairwise resistances between nodes.

For an example of the Theorem, if there are 6 nodes, then
\begin{align*}
 \det L_{(1,2,3),(3,4,5)}
 =& \pu{15|24|6} - \pu{14|25|6} \\
 =& \pu{153|24|6} + \pu{15|243|6} + \pu{15|24|63} \\& - \pu{143|25|6} - \pu{14|253|6} - \pu{14|25|63},
\end{align*}
which for circular planar graphs is just
$\pu{15|243|6}$.

When $C=\emptyset$, this determinant formula is
equivalent to Lemma~4.1 of Curtis-Ingerman-Morrow \cite{\CIM},
though their formulation is a bit more complicated.
The formula $L_{i,j} = \pu{i,j|\text{rest singletons}}$ \cite[Proposition~2.8]{KW:polynomial} is a further specialization, with $A=\{i\}$ and $B=\{j\}$.

\begin{proof}[Proof of Theorem.]
We assume first that $L_{C,C}$ is non-singular.
By standard linear algebra
\begin{align}
\det L_{(A \cup C), (B \cup C)}
= \begin{vmatrix}L_{A,B}& L_{A,C} \\ L_{C,B}& L_{C,C}\end{vmatrix}
&=  \begin{vmatrix}L_{A,B}-L_{A,C}L_{C,C}^{-1}L_{C,B} & L_{A,C}L_{C,C}^{-1}\\0&I\end{vmatrix}
  \begin{vmatrix}I&0\\L_{C,B}&L_{C,C}\end{vmatrix} \notag\\
   &= \det [ L_{A,B} - L_{A,C} L_{C,C}^{-1} L_{C,B} ] \det L_{C,C} \label{schur}
\end{align}
Since this is essentially Schur reduction, $  L_{A,B} - L_{A,C} L_{C,C}^{-1} L_{C,B}
$
is the $A,B$ submatrix of the response matrix when nodes in $C$ are
redesignated as internal, so by Lemma~4.1 of Curtis-Ingerman-Morrow \cite{\CIM},
\begin{equation} \label{schur-1}
\det [ L_{A,B} - L_{A,C} L_{C,C}^{-1} L_{C,B} ] =
\frac{\text{signed sum of pairings from $A$ to $B$ when $C$ is internal}}{
\text{uncrossing when $C$ is internal}}.
\end{equation}

If we glue the nodes not in $C$ together, the response matrix of the
resulting graph has $L_{C,C}$ as a co-dimension $1$ submatrix, so by
Lemma~A.1 of Kenyon-Wilson \cite{KW:polynomial},
\begin{equation} \label{schur-2}
 \det L_{C,C} =
 \frac{\text{forests rooted at $A\cup B\cup D$}}{\text{uncrossing}} =
 \frac{\text{uncrossing when $C$ is internal}}{\text{uncrossing}}.
\end{equation}
Combining Equations~\ref{schur}, \ref{schur-1}, and~\ref{schur-2}
gives the result for nonsingular $L_{C,C}$.

The case of singular $L_{C,C}$ can be obtained as a limit of the
above nonsingular case.
\end{proof}

\section{Carroll-Speyer groves}\label{CSgroves}

Here we study the groves of Carroll and Speyer.  For Carroll and
Speyer's work, the relevant graph is a triangular portion of the
triangular grid, shown in \fref{csgrove}.  Carroll and Speyer computed
the number of groves on this graph which form a tripod grove (for $N$
even) or a tripartite grove (for $N$ odd).  The relevant tripod or
tripartite partition is the one for which the three sides of the
triangular region form the three color classes, and each part connects
nodes with different colors.  For the case $N=6$, the relevant tripod
partition is $1,17|2,16|3,9,15|4,8|5,7|10,14|11,13|6|12|18$, and an
example grove is shown in \fref{csgrove}.  The number of such groves
turns out to always be a power of $3$, specifically, when there are
$3N$ nodes, there are $3^{\lfloor N^2/4\rfloor}$ groves.  In this
section we consider these graphs as a test case for our methods for
counting groves.  There is much that we can compute, but we do not
know how at present to obtain a second derivation of the power-of-$3$
formula.

\begin{figure}[b!]
\centerline{\includegraphics{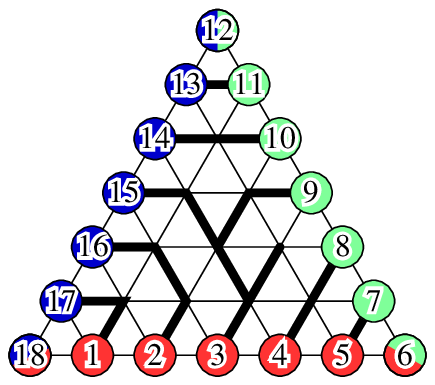}\hfil\includegraphics{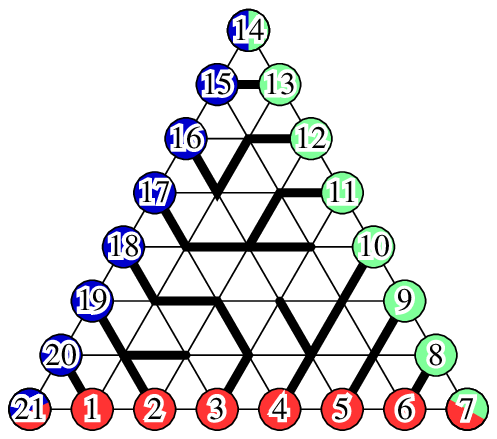}}
\caption{
  Carroll-Speyer graph for $N=6$ (left) and $N=7$ (right), each shown
  with one of its Caroll-Speyer groves.  The graphs have $n=3N$ nodes;
  for even~$N$ the grove partition type is a (generalized) tripod,
  while for odd~$N$ the grove partition type is a (generalized)
  tripartite pairing.  The number of Carroll-Speyer groves is
  $3^{\lfloor N^2/4\rfloor}$ \cite{\CS}.  \label{csgrove}}
\end{figure}

We need the entries of the $L$ matrix in order to compute the
connection probabilities using the Pfaffian and Pfaffianoid formulas
presented in \sref{dualtripodL} and \sref{tripodL}.  To compute the
tripartite connection probabilities, we need those entries of the $L$
matrix whose indices come from different sides of the triangle.  From
symmetry considerations, it suffices to consider the entries between
the first two sides.  In the $N=6$ example from \fref{csgrove}, the
submatrix of the $L$ matrix with rows indexed by the nodes on side~1
(excluding corners) and columns indexed by the nodes on side~2
(excluding corners) is given by
\begin{align*}
\begin{bmatrix}
 L_{1,7} & L_{1,8} & L_{1,9} & L_{1,10} & L_{1,11} \\
 L_{2,7} & L_{2,8} & L_{2,9} & L_{2,10} & L_{2,11} \\
 L_{3,7} & L_{3,8} & L_{3,9} & L_{3,10} & L_{3,11} \\
 L_{4,7} & L_{4,8} & L_{4,9} & L_{4,10} & L_{4,11} \\
 L_{5,7} & L_{5,8} & L_{5,9} & L_{5,10} & L_{5,11}
\end{bmatrix}
&=
\begin{bmatrix}
 \frac{31}{9456} & \frac{25}{2364} & \frac{23}{1576} & \frac{25}{2364} & \frac{31}{9456} \\
 \frac{37}{2364} & \frac{445}{9456} & \frac{23}{394} & \frac{355}{9456} & \frac{25}{2364} \\
 \frac{87}{1576} & \frac{53}{394} & \frac{97}{788} & \frac{23}{394} & \frac{23}{1576} \\
 \frac{529}{2364} & \frac{3043}{9456} & \frac{53}{394} & \frac{445}{9456} & \frac{25}{2364} \\
 \frac{11167}{9456} & \frac{529}{2364} & \frac{87}{1576} & \frac{37}{2364} & \frac{31}{9456}
\end{bmatrix}
\\&=
\begin{bmatrix}
 \phantom{-2015}\mathllap{1}     & \phantom{-2015}\mathllap{-1}    & \phantom{-208}\mathllap{1}    & \phantom{-31}\mathllap{-1} & \fm 1 \\
 \phantom{-2015}\mathllap{-31}   & \phantom{-2015}\mathllap{24}    & \phantom{-208}\mathllap{-16}  & \phantom{-31}\mathllap{8} & -1 \\
 \phantom{-2015}\mathllap{361}   & \phantom{-2015}\mathllap{-208}  & \phantom{-208}\mathllap{81}   & \phantom{-31}\mathllap{-16} & \fm 1 \\
 \phantom{-2015}\mathllap{-2015} & \phantom{-2015}\mathllap{888}   & \phantom{-208}\mathllap{-208} & \phantom{-31}\mathllap{24} & -1 \\
 \phantom{-2015}\mathllap{5297}  & \phantom{-2015}\mathllap{-2015} & \phantom{-208}\mathllap{361}  & \phantom{-31}\mathllap{-31} & \fm 1
\end{bmatrix}
^{-1}.
\end{align*}
We explain here why the inverse of this submatrix is integer-valued,
and how to interpret the entries.

Recall \tref{Lminorthm} on minors of the $L$ matrix.  Letting $A$, $B$, and
$C$ denote the nodes on the first, second, and third sides respectively,
$$\det[L_{i,j}]^{i\in A}_{j\in B} = \frac{Z(\text{nodes of $A$ paired with nodes of $B$, nodes of $C$ singletons})}{Z(\unc)} = \frac{1}{Z(\unc)}.
$$
Likewise
\newcommand{\nearfull}{Z\left(\genfrac{}{}{0pt}{0}{\text{nodes of $A\setminus\{i_0\}$ paired with nodes of $B\setminus\{j_0\}$,}}{\text{nodes of $C\cup\{i_0,j_0\}$ singletons}}\right)}
$$\det[L_{i,j}]^{i\in A\setminus\{i_0\}}_{j\in B\setminus\{j_0\}} = \frac{\nearfull}{Z(\unc)}.
$$
Thus the $i_0,j_0$ entry of the inverse of the above matrix is
$$ \left([L_{i,j}]^{i\in A}_{j\in B}\right)^{-1}_{j_0,i_0} = (-1)^{i_0+j_0} \nearfull.$$
When there are edge weights, the $i_0,j_0$ entry of the inverse matrix
will be given by the corresponding polynomial in the edge weights.

To get the normalized probability of the tripartite partition (for odd~$N$), the
Pfaffian we need is
$$
\Pf \begin{bmatrix}
\fm 0 & \fm L_{A,B} & \fm L_{A,C} \\
- L_{B,A} & \fm 0 & \fm L_{B,C} \\
- L_{C,A} & -L_{C,B} & \fm 0
\end{bmatrix}
=
\Pf \begin{bmatrix}
\fm 0 & \fm L_{A,B} & \fm L_{A,B}^T \\
- L_{A,B}^T & \fm 0 & \fm L_{A,B} \\
- L_{A,B} & -L_{A,B}^T & \fm 0
\end{bmatrix}
$$
which in the case $N=7$ gives
$\pu{\text{tripartite grove}}=531441/135418115000$.
The calculations for the tripod partitions for even~$N$ is similar,
except that we take a Pfaffianoid rather than a Pfaffian.

To compute the number (as opposed to probability) of groves of a given type, we also need the number of
spanning forests rooted at the nodes.  The number of spanning forests
may be computed from the graph Laplacian using the matrix-tree theorem,
which yields the following formula
$$
\prod_{\substack{\{\alpha,\beta,\gamma\}\\ \alpha^{3N)}=1 \\ (\alpha/\beta)^{N}=1 \\ \alpha\beta\gamma=1 \\ \alpha,\beta,\gamma\text{\ distinct}}} (6-\alpha-\alpha^{-1}-\beta-\beta^{-1}-\gamma-\gamma^{-1})
$$
(see \cite[\S~6.9]{KPW} for the derivation of a similar formula).  In the case
$N=7$ this formula yields $135418115000$, so there are $531441 =
3^{12} = 3^{\lfloor 7^2/4\rfloor}$ tripartite groves, in agreement
with Carroll and Speyer's formula.  Is it possible to derive the
$3^{\lfloor N^2/4\rfloor}$ formula using this approach?

\appendix

\section{Pfaffianoid formula for tripod partitions}\label{sec:pfd}

\begin{proof}[Proof of \tref{thm:tripod}]
  Any column partition $\tau$ contributing to $\sigma$ will have
  $(n-1)/2$ parts (as $\sigma$ does) and no singleton parts, and as such
  it will consist of a single tripleton part together with doubleton
  parts.  To determine what $\tau$ contributes to $\sigma$, we may use
  the following abbreviated rules.  For two crossing
  doubletons, as in \eqref{pairing-rule} we
  use
$$  13|24 \to - 12|34 - 14|23.$$

For a crossing doubleton and a tripleton, recall (Lemmas~\ref{cross}
and \ref{split}) that
\begin{align*}
  13|24  &\equiv 1|234  +  2|134  +  3|124  +  4|123  -  12|34  -  14|23 \\
  13|245 &\equiv 1|2345 +  2|1345 +  3|1245 +  45|123 -  12|345 -  145|23
\intertext{so we may use the rule}
  13|245 &\to 45|123 - 12|345 - 23|154
\end{align*}

After we apply the $13|245$ rule, let us consider another doubleton
part.  If the doubleton part did not cross $13$ or $245$, it will
cross none of $45$, $123$, $12$, $345$, $23$, or $154$.  Otherwise it
is one of the following forms:
\begin{center}
\begin{tabular}{c| c@{$|$}c c c@{$|$}c c@{$|$}c c@{$|$}c}
        crosses&13&245 &$\to$&   45&123 & 12&345 & 23&154 \\
\hline
 .5, 1.5       &y &n   &&  n &y   & y &n   & n &y   \\
 .5, 2.5       &y &y   &&  n &y   & n &n   & y &y   \\
 .5, 3.5       &n &y   &&  n &n   & n &y   & n &y   \\
 .5, 4.5       &n &y   &&  y &n   & n &y   & n &y   \\
1.5, 2.5       &n &y   &&  n &y   & y &n   & y &n   \\
1.5, 3.5       &y &y   &&  n &y   & y &y   & n &n   \\
1.5, 4.5       &y &y   &&  y &y   & y &y   & n &y   \\
2.5, 3.5       &y &n   &&  n &y   & n &y   & y &n   \\
2.5, 4.5       &y &y   &&  y &y   & n &y   & y &y   \\
3.5, 4.5       &n &y   &&  y &n   & n &y   & n &y   \\
\end{tabular}\end{center}
In each case the doubleton part crosses at most as many other parts in
the new partitions as it did in the old partition.  Thus the number of
crossing parts strictly decreases.

After we apply the $13|24$ rule, we saw already that the number of
crossings amongst doubleton parts strictly decreases.  Now let us
consider crossings amongst doubleton parts and the tripleton part.
The tripleton part divides the vertices into three sectors.  The
distribution of $1,2,3,4$ amongst these sectors is one of
\begin{itemize}
\item
  All four endpoints in same sector.
    Before rule neither of the chords 13,24 cross the tripleton,
    nor do any of them after the rule.
\item
  Three endpoints in same sector.
    Before rule one chord crosses tripleton,
    after rule one chord crosses tripleton.
\item
  At least one sector contains exactly two endpoints.
    If we apply the rule then it must be that the chords crossed,
    so both chords cross the tripleton.
    After rule, in one partition two chords cross tripleton, in other
    partition neither chord crosses tripleton.
\end{itemize}
In any event, the total number of crossing parts strictly decreases.

We may repeatedly apply the $13|245$ rule and $13|24$ rule, in any order,
until no two parts cross, and we are guaranteed to obtain a linear
combination of planar partitions containing a tripleton part and rest
doubleton parts.

Recall the rule
 $$ 13|245 \to 45|123 - 12|345 - 23|154.$$
Suppose that at least one of these parts has two or more vertices of
the same color, say red.  We have the following possibilities:
\renewcommand{\R}[1]{\overset{r}{#1}}
\begin{align*}
  \R1\R3|\R245 &\to 45|\R1\R2\R3 - \R1\R2|\R345 - \R2\R3|\R154
\\
  13|2\R4\R5 &\to \R4\R5|123 - 12|3\R4\R5 - 23|1\R5\R4
\\
  1\R3|\R2\R45 &\to \R45|1\R2\R3 - 1\R2|\R3\R45 - \R2\R3|15\R4
\\
  \R13|\R24\R5 &\to 4\R5|\R1\R23 - \R1\R2|34\R5 - \R23|\R1\R54
\end{align*}
(the remaining cases contain more reds than these).
In each case, each of the resulting partitions has a part with two or more
vertices of the same color, so by induction, such partitions will
not contribute to $\sigma$.  Thus we may restrict attention to partitions
$\tau$ with bichromatic doubleton parts and a trichromatic tripleton part.
\renewcommand{\R}{\mathbb{R}}

We take another look at the rule
$$13|245 \to 45|123 - 12|345 - 23|154.$$
Two colors occur twice.  There are three possibilities (up to
renamings of color):
\begin{center}\begin{tabular}{lcl}
& 5,1 red and 3,4 blue & in which case we may use the rule $13|245 \to 45|123$   \\
($*$)&  5,1 red and 2,3 blue & in which case we may use the rule $13|245 \to - 12|345$ \\
 ($**$)& 1,2 red and 3,4 blue & in which case we may use the rule $13|245 \to - 23|154$
\end{tabular}\end{center}
(And similarly for the $13|24$ rule, either $13|24 \to -12|34$ or $13|24 \to -14|23$.)

Thus each maximally multichromatic partition $\tau$ contributes either
$+1$ or $-1$ to $\sigma$.

Let us compare the contributions to $\sigma$ of partitions~$\tau$ which
contain doubleton parts that do not cross one another, and a tripleton
part which may cross the doubleton parts.  Any such partition $\tau$
may be obtained from $\sigma$ by repeatedly transposing items of the
tripleton with their neighbors.  From the above rules ($*$)
and ($**$) we see that each
such transposition changes the sign of $\tau$'s contribution to
$\sigma$.

Next we consider what happens if we leave the tripleton part alone and
only apply the $13|24$ rule.  Upon deleting the tripleton part, and
recalling our earlier analysis of the tripartite pairing
(\tref{tripartite}), we see that
each partition $\tau$ contributes to $\sigma$ a sign which involves
moving the tripod to its desired location in $\sigma$, times the sign in
the corresponding Pfaffian in which rows and columns
of items of tripleton part have been deleted.  This gives us the
tripod Pfaffianoid formula:
\begin{multline*}
(-1)^\text{sum of items in $\sigma$'s tripleton part} \times \\
  \sum_{a \in R, b \in G, c \in B}
       (-1)^{a+b+c} (L_{a,b} L_{b,c} + L_{a,b} L_{a,c} + L_{a,c} L_{b,c}) \Pf(R\smallsetminus\{a\};G\smallsetminus\{b\};B\smallsetminus\{c\}).
\end{multline*}
\end{proof}

\section{Tripartite  pairings in terms of Pfaffians in the resistances}
\label{dualtripodR}

There is a formula analogous to the one in \tref{tripartite} for the
probability of a tripartite partition, expressing it in terms of
the pairwise electrical resistances~$R_{i,j}$ rather than the response
matrix entries~$L_{i,j}$.  The normalization is also different, rather
than normalizing by $\Pr(1|2|\cdots|n)$, we normalize by $\Pr(12\cdots n)$.
The corresponding formula is illustrated by the following example:
\begin{multline}%\everymath{\scriptstyle}
\frac{\Pr(16|23|45)}{\Pr(123456)} = \\
[t]\Pf\!\left[\begin{matrix}
 0         &  0         &\plt-\frac12 R_{1,3} &\plt-\frac12 R_{1,4} &\plt-\frac12 R_{1,5} &\plt-\frac12 R_{1,6} \\
 0         &  0         &\plt-\frac12 R_{2,3} &\plt-\frac12 R_{2,4} &\plt-\frac12 R_{2,5} &\plt-\frac12 R_{2,6} \\
-t+\frac12 R_{1,3} & -t+\frac12 R_{2,3} &  0         &  0         &\plt-\frac12 R_{3,5} &\plt-\frac12 R_{3,6} \\
-t+\frac12 R_{1,4} & -t+\frac12 R_{2,4} &  0         &  0         &\plt-\frac12 R_{4,5} &\plt-\frac12 R_{4,6} \\
-t+\frac12 R_{1,5} & -t+\frac12 R_{2,5} & -t+\frac12 R_{3,5} & -t+\frac12 R_{4,5} &  0         &  0         \\
-t+\frac12 R_{1,6} & -t+\frac12 R_{2,6} & -t+\frac12 R_{3,6} & -t+\frac12 R_{4,6} &  0         &  0
\end{matrix}\right]
\end{multline}
Here the Pfaffian is a polynomial in $t$, in fact a linear function of
$t$, and the coefficient of the linear term gives $\Pr(16|23|45)/\Pr(123456)$.

We prove here this formula when one of the color classes is empty, so
that the Pfaffian is actually a determinant (see \tref{LtoR}).  The
general tripartite Pfaffian formula follows from the bipartite
determinant special case, together with a result that we prove in our
next article that for any planar pairing, the $L$-polynomial is a
linear combination of such determinants \cite{KW-marginals}.

\newcommand{\iL}{{\tilde L}^{-1}}

Recall that any codimension-1 minor of $L$ is the ratio of the
spanning trees to the uncrossing.  Let $\tilde L$ be the $(n-1)\times
(n-1)$ submatrix of $L$ obtained by deleting the last row and column.
Recall that $R_{i,j}=\iL_{i,i}+\iL_{j,j}-2\iL_{i,j}$, where
$\iL_{i,j}$ is interpreted to be $0$ if either $i=n$ or $j=n$
(see \cite[Section~A.2]{KW:polynomial}).

\begin{lemma} \label{LR}
  For each row of $I+L R/2$, the row's entries are all the same.
\end{lemma}
\begin{proof}
Suppose $i\neq n$.  Then
\begin{align*}
\sum_{j} L_{i,j} R_{j,k}
 &=
\sum_{j} L_{i,j} [\iL_{j,j}+\iL_{k,k}-2\iL_{j,k}]\\
 &=
\sum_{j} L_{i,j} [R_{j,n}+0]-2\delta_{i,k}.
\end{align*}
So for any $i\neq n$, the $i$\th row of $I+L R/2$ is constant.  By
re-indexing the rows, we see that this must hold for the $n$\th row as
well.
\end{proof}

\begin{lemma}
  The diagonal entries of $R L R$ are all the same.
\end{lemma}
\begin{proof}
We have
\begin{align*}
\sum_{j,k} R_{i,j} L_{j,k} R_{k,i}
 &=
\sum_{j,k} (\iL_{i,i}+\iL_{j,j}-2\iL_{i,j}) L_{j,k} (\iL_{i,i}+\iL_{k,k}-2\iL_{i,k}).
\end{align*}
{The third factor contains no $j$ subscripts, and neither does the $\iL_{i,i}$ term of the first factor, so summing over $j$ gives $0$, and we may drop the $\iL_{i,i}$ term in the first factor.  Similarly, we may drop the $\iL_{i,i}$ term in the third factor.}
\begin{align*}
\sum_{j,k} R_{i,j} L_{j,k} R_{k,i}
 &=
\sum_{j,k} (\iL_{j,j}-2\iL_{i,j}) L_{j,k} (\iL_{k,k}-2\iL_{i,k})\\
 &=
\sum_{j,k} \iL_{j,j} L_{j,k} \iL_{k,k}
-2\sum_{j,k} \iL_{j,j} L_{j,k} \iL_{i,k}
-2\sum_{j,k} \iL_{i,j} L_{j,k} (\iL_{k,k}-2\iL_{i,k})\\
&=
\sum_{j,k} \iL_{j,j} L_{j,k} \iL_{k,k}
-2\sum_{j} \iL_{j,j} \delta_{j,i}
-2\sum_{k} \delta_{i,k} (\iL_{k,k}-2\iL_{i,k})\\
&=
\sum_{j,k} \iL_{j,j} L_{j,k} \iL_{k,k}
\end{align*}
which in particular does not depend upon $i$: the diagonal terms of $R
L R$ are all equal.  \end{proof}

Suppose that we add an extra $(n+1)$\st node to the graph where the
conductance between nodes $j$ and $n+1$ is $\eps(1+\frac12\sum_k
L_{j,k} R_{j,k})$, where $\eps\approx 0$.  (Some of these conductances
may be negative.)  From \lref{LR} we have
\begin{align*}
\delta_{i,j}+\frac12\sum_k L_{j,k} R_{k,i} &=
1+\frac12 \sum_k L_{j,k} R_{j,k} \\
1 = \sum_j\left[\delta_{i,j}+\frac12\sum_k L_{j,k} R_{k,i}\right] &=
\sum_j\left[1+\frac12 \sum_k L_{j,k} R_{j,k}\right],
\end{align*}
so adding up the conductances of these new edges gives $\eps$.
If we set node $n+1$ to
be at $1$ volt and node $i$ to be at $0$ volts, then to first order, each of
the nodes $1,\dots,n$ be nearly at $0$ volts.  Then the current flowing
from node $n+1$ to node $j$ is
$(1+o(1))\eps(\delta_{i,j}+\frac12\sum_k L_{j,k} R_{k,i})$, and the
total current from node $n+1$ is $(1+o(1))\eps$.  Let $v$ denote the
voltages of the first $n$ nodes.  Now we view the first $n$ nodes as
a circuit to which we apply the voltages~$v$.  The current flowing into
node $j\neq i$ is $(1+o(1))\frac\eps 2\sum_k L_{j,k} R_{k,i}$, and
the current flowing into node $i$ is
$(1+o(1))\eps(1+\frac12\sum_k L_{i,k} R_{k,i})-(1+o(1))\eps$
(the first term is the current from the $(n+1)$\st node to
node~$i$, and the second term is the current flowing out of node~$i$,
 i.e., the total current flowing in from node~$(n+1)$).
To summarize, we have $$L v = (1+o(1))\frac12\eps L
R \delta_i,$$ where $\delta_i$ is the $i$\th basis vector.
We would like to cancel the $L$'s from this equation,
but the equation only determines $v$ up to an additive constant:
$v=\frac12 \eps R \delta_i+o(\eps)+\text{const}$.
But $v_i=0$ and $R_{i,i}=0$, so the additive constant is zero.
Knowing the voltages $v_j=\frac12 \eps R_{i,j}+o(\eps)$ allows
us to compute the currents to second order: the total current from
node~$n+1$ is
\begin{align*}
\sum_j \eps \left[\delta_{i,j} + \frac12\sum_k L_{j,k} R_{k,i}\right] \left(1-\frac\eps2 R_{i,j} + o(\eps)\right)
&= \eps - \frac{\eps^2}{4} \sum_{j,k} R_{i,j} L_{j,k} R_{k,i} + o(\eps^2)\\
&= \eps - \frac{\eps^2}{4} Q + o(\eps^2)
\end{align*}
where $Q$ is the diagonal entry of $RLR$.  Let $\tilde R$
denote the resistances in the extended graph; we have
$$ \tilde R_{i,n+1} = 1/\eps + Q/4 + o(1).$$

Let $\tilde L$ be the $L$-matrix of the enlarged graph, $\tilde L'$ be
the submatrix obtained by deleting the $(n+1)$\st vertex, and $\tilde
L'^{-1}$ be the inverse of $\tilde L'$ (sometimes extended to have an
all-$0$ $(n+1)$\st row and column).  The ratio of spanning trees to the
uncrossing is
\begin{align*}
 \lim_{\eps\to0} \frac1\eps \det (-\tilde L')
&= \lim_{\eps\to0} \frac1\eps \frac{1}{\det (-\tilde L'^{-1})} \\
\intertext{now recall $(-\tilde L'^{-1})_{i,j} = (\tilde R_{i,n+1}+\tilde R_{j,n+1}-\tilde R_{i,j})/2$, so}
&= \lim_{\eps\to0} \frac1\eps \frac{1}{\det [1/\eps + Q/4 - R_{i,j}/2 + o(1)]^{i=1,\dots,n}_{j=1,\dots,n}}. \\
\intertext{This determinant is a linear function of $t=1/\eps+Q/4$, so we may rewrite this limit as}
&= \frac{1}{[t]\det [t - {\textstyle\frac12}R_{i,j}]^{i=1,\dots,n}_{j=1,\dots,n}}.
\end{align*}

Let $A$ and $B$ be two subsets of the nodes.  We have
\begin{align}
\det[L_{i,j}]^{i\in A}_{j\in B}
 =& \lim_{\eps\to0} \det[\tilde L'_{i,j}]^{i\in A}_{j\in B}\nonumber \\
 =& \lim_{\eps\to0} \det\tilde L' \det[\tilde L'^{-1}_{j,i}]^{j\in B^c}_{i\in A^c}
 \times (-1)^{\sum_{a\in A}a+\sum_{b\in B} b} \nonumber \\
 =& \lim_{\eps\to0} \frac{\eps}{[t]\det [t - {\textstyle\frac12}R_{i,j}]^{i=1,\dots,n}_{j=1,\dots,n}} \det[-1/\eps - Q/4 + R_{i,j}/2 + o(1)]^{j\in B^c}_{i\in A^c}\nonumber\\
 &\ \ \ \ \ \times (-1)^{\sum_{a\in A}a+\sum_{b\in B} b} \nonumber\\
 =& \frac{[t]\det[-t + {\textstyle\frac12}R_{i,j}]^{j\in B^c}_{i\in A^c}}{[t]\det [t - {\textstyle\frac12}R_{i,j}]^{i=1,\dots,n}_{j=1,\dots,n}}
\times (-1)^{\sum_{a\in A}a+\sum_{b\in B} b},\label{Rform1}
\end{align}
where in the above, the rows and columns in~$A$, $B$, $A^c$, and~$B^c$
are arranged in sorted order.

Note that \eqref{Rform1} allows us to rewrite any minor of the $L$
matrix in terms of the pairwise resistances between the nodes.
Since the denominator of the
right-hand side of \eqref{Rform1} is $Z(\unc)/Z(\tree)$, we have
\begin{equation}\label{Rform}
\frac{Z(\unc)}{Z(\tree)}\det[L_{i,j}]^{i\in A}_{j\in B}
 = [t]\det[-t + {\textstyle\frac12}R_{i,j}]^{j\in B^c}_{i\in A^c}
\times (-1)^{\sum_{a\in A}a+\sum_{b\in B} b}.
\end{equation}
If $A=B^c$, then \eqref{Rform} simplifies to yield
\begin{theorem} \label{LtoR}
  If $A$ and $B$ are disjoint and equinumerous sets of nodes and
  $A\cup B$ is the set of all nodes, then
$$
\frac{Z(\unc)}{Z(\tree)}\det[L_{i,j}]^{i\in A}_{j\in B}
 = [t]\det[t - {\textstyle\frac12}R_{i,j}]^{i\in A}_{j\in B}.
$$
\end{theorem}

\pdfbookmark[1]{References}{bib}
\bibliographystyle{hmralpha}
\bibliography{bc}

\end{document}